\documentclass[smallextended,envcountsect,envcountsame]{svjour3}
 \usepackage[margin=1.1in]{geometry} 
\usepackage{graphicx}
\usepackage{mathrsfs}
\usepackage{fix-cm}
\usepackage{amsmath}
\usepackage{amssymb}
\usepackage{latexsym}
\usepackage{dsfont}
\usepackage{xcolor}
\usepackage{cite}
\usepackage{dsfont}
\usepackage{enumitem}

\usepackage{multicol}
\usepackage{amsmath}
\usepackage{multicol}
\usepackage{amssymb}
\usepackage{cite}
\usepackage{enumitem}
\usepackage{mathtools}
\usepackage{dsfont}

\newcommand{\R}{\mathbb{R}}
\newcommand{\Rex}{\overline{\mathbb{R}}}
\def\tto{\rightrightarrows} 
 
\let\epsilon\epsilon 
\newcommand{\N}{\mathbb{N}}
\DeclareMathOperator{\cl}{cl}

\DeclareMathOperator{\epi}{epi}

\DeclareMathOperator{\dom}{\operatorname{dom}}

\DeclareMathOperator{\inte}{int}

\DeclareMathOperator{\supp}{\textnormal{supp}}

\newcommand{\Constr}{\mathcal{Q}}
\newcommand{\C}{\mathcal{C}}
\newcommand{\asub}{{\tilde{\partial}}}

\newcommand{\Radon}[1]{{\mathcal{M}}_+(#1)}
\newcommand{\Intf}[1]{\mathcal{I}_{#1}}
\newcommand{\Borel}[1]{{\mathcal{B}}(#1)}
\newcommand{\Sp}{ \mathbb{S}^p}
\DeclareMathOperator{\Tr}{tr}
\let\epsilon\varepsilon
\let\subseteq\subset
\let\supseteq\supset
 \usepackage{lipsum}
 \usepackage{amsfonts}
 \usepackage{graphicx}
 \usepackage{epstopdf}
 \usepackage{algorithmic}
 \usepackage{enumitem}
 \ifpdf
 \DeclareGraphicsExtensions{.eps,.pdf,.png,.jpg}
 \else
 \DeclareGraphicsExtensions{.eps}
 \fi
 
\begin{document}
\title{Optimality conditions in     DC constrained mathematical programming problems
\thanks{The first author was  partially supported by ANID grant Fondecyt Regular 1190110 and Centro de Modelamiento Matemático (CMM), ACE210010 and FB210005, BASAL funds for centers of excellence from ANID-Chile. The second author has been partially supported by Grants MTM2014-59179-C2-(1-2)-P from MINECO/MICINN, Spain, and FEDER , European Union, and by the Australian Research Council, Project DP180100602. The third author was partially supported by ANID grants  Fondecyt Regular 1190110 and  Fondecyt Regular 1200283 and Centro de Modelamiento Matemático (CMM), ACE210010 and FB210005, BASAL funds for centers of excellence from ANID-Chile.}}
\author{Rafael Correa\thanks{Universidad de O'Higgins, Rancagua,  Chile and DIM-CMM of Universidad de Chile, Santiago, Chile 
		(\email{rcorrea@dim.uchile.cl}).}
	\and Marco A. L\'opez  \thanks{Department of Mathematics, University of Alicante, 03071 Alicante, Spain, and CIAO, Federation University, Australia
		(\email{marco.antonio@ua.es}).}
	\and Pedro P\'erez-Aros\thanks{Instituto de Ciencias de la Ingenier\'ia, Universidad de O'Higgins, Rancagua, Chile 
		(\email{pedro.perez@uoh.cl}).} }

\date{Received: date / Accepted: date}

 \maketitle

\begin{abstract}
 
	This paper provides necessary and sufficient optimality conditions for
	abstract constrained mathematical programming problems in locally convex
	spaces under new qualification conditions. Our approach
	exploits the geometrical properties of certain mappings, in particular their
	structure as  difference of convex functions, and uses techniques of
	generalized differentiation (subdifferential and coderivative). It turns 	out that these tools can be  used fruitfully out of the scope of Asplund
	spaces. Applications to infinite, stochastic and semi-definite
	programming are developed in separate sections. 
\end{abstract}
 \vspace*{-0.05in}
	
\vspace*{-0.05in}

\vspace*{-0.05in}

\section{Introduction}

Mathematical programming has been recognized as one of the
fundamental chapters of applied mathematics since a
huge number of problems in engineering, economics, 
management science, etc., involve an optimal decision-making
process which gives rise to an optimization model. This  fact  
has intensely motivated the  theoretical foundations
of optimization and the study and development of algorithms.

Among the main issues in mathematical optimization, optimality
conditions play a key role in the theoretical understanding of solutions and
their numerical computation. At first such conditions were
set for linear or smooth optimization problems.   Later developments in variational analysis allowed researchers to extend this theory to general nonlinear nonsmooth convex 
programming problems defined in infinite-dimensional spaces (see, e.g., \cite{MR1921556}). In
the same spirit, generalized differentiation properties became an 
intensive field of research with numerous applications to nonsmooth and 
nonconvex mathematical programming problems. Nevertheless, in order to provide such general calculus rules and compute optimal conditions, a certain degree of smoothness is required, while working either in Banach spaces with a smooth norm or in Asplund spaces  (see, e.g., \cite{MR2191744,MR2191745,MR1491362,MR3823783,MR4435762}). In this paper, we provide
necessary and sufficient optimality conditions for a general class of
optimization problems under new qualification conditions which constitute  
real alternatives to the well-known Slater constraint qualification. Our
approach is based on the notions of the (regular) subdifferential and
coderivative and we show that these tools work out of the
scope of Asplund spaces, not for the whole family of lower semicontinuous
functions, but for the so-called class of B-DC mappings. 
This class of mappings is introduced in Definition \ref{definition_DCfunction}, and constitutes a slight  extension of 	the concept of DC functions/mappings.

The paper is focused on the study of (necessary and
sufficient) optimal conditions for a constrained programming
problem. First, we study the case where the constraint  is of  the
form, $\Phi (x)\in \mathcal{C}$, where $\mathcal{C}$ is a nonempty, closed
and convex set in a vector space $Y$, and $\Phi $ is a  vector-valued
function from the decision space $X$ into $Y$. Second, we study an
abstract conic constraint, that is, the case when $\mathcal{C}=-\mathcal{K}$, for a nonempty  convex closed cone $\mathcal{K}$. 
These abstract representations allow us to cover infinite, stochastic,
and semidefinite programming problems.

With this general aim of establishing necessary and sufficient
optimality conditions, we first  introduce  an extension of the
concept of  vector-valued DC mappings, also called $\delta$-convex mappings,
given in \cite{MR1016045} (see also \cite{MR3785670} for classical notions
and further references). Our Definition \ref{definition_DCfunction} 
addresses two fundamental aspects in mathematical optimization. First, the convenience   of using 		functions with extended real values and mappings which are not defined in
the whole space  	has been widely recognized. This allows us  to handle different classes of
constraint systems from an abstract point of view and, to this
purpose, we enlarge the space $Y$ with an extra element $\infty
_{Y}$ (in particular, $\infty _{Y}=+\infty $, whenever $Y=		\mathbb{R}$).   Second, we consider specific scalarization sets for the
mapping $\Phi $, which varies along dual directions in  respect to the
set involved in the constraint; more specifically, directions in
the polar set of $\mathcal{C}$, or in the positive polar cone of $\mathcal{K}$, respectively (see definitions below). 

The aforementioned  notions make it possible for us to exploit  
the geometrical properties of mappings (convexity) combined with tools 
taken from variational analysis and generalized differentiation.  Using these tools, we obtain
necessary and sufficient optimality conditions of   abstract 
optimization problems defined on general locally convex spaces under new
qualification conditions, which represent alternatives  for  classical notions.

The paper is organized as follows. In Section \ref{notation}
we introduce the necessary notations and preliminary results; we 
give   the definition of the set $\Gamma _{h}(X,Y)$ and 
introduce the tools of generalized differentiation. Together  with those
notions, we provide the first calculus rules and machinery needed
in the paper, which constitute the working horse in the  
formulation  of our optimal conditions.   In Section \ref{SECT_DCMATHPROGR}  we deal  with   a constraint  of the form $\Phi (x)\in \mathcal{C}$; we transform the problem into an
unconstrained mathematical program, where the objective function is a
difference of two convex functions, and this reformulation yields 
necessary and sufficient conditions of global and local optimality. The
main result in this section, concerning global optimality, is
Theorem \ref{THEo:DC01}, and the result for necessary conditions of 
local optimality is Theorem \ref{theo:local:opt}; meanwhile sufficient
conditions are given in Theorem \ref{suf_local_conv}. Later in
Section \ref{SECTION_CONE}, we confine ourselves to  studying
problems with abstract conic constraints  given by  $\Phi (x)\in -\mathcal{K%
}$. In such Section, the cone structure is exploited, and a set of
scalarizations,  generating by the positive polar cone  $\mathcal{K}^+$
(see Definition \ref{defweakcompt}), is used. Appealing to  	that notion, and thanks to a suitable reformulation of the
problem, we   derived  specific necessary and sufficient optimality
conditions for conic programming problems. In particular, Theorem %
\ref{conecons_global} presents global optimality conditions 
and Theorems \ref{conecons_local} and \ref{suf_local_cone} are
devoted to local optimality. In the final section, we   apply our developments to establish  \emph{ad hoc} optimality conditions
for fundamental problems in applied mathematics such as infinite,
stochastic and semidefinite programming problems.

\section{Notation and preliminary results}
\label{notation}
The paper uses the main notations and definitions  which are standard in convex and variational analysis (see, e.g., \cite{MR2191744,MR2191745,MR1491362,MR3823783,MR4435762,MR1921556,MR3890045}).  
\subsection{Tools form convex analysis}
In this paper $X$ and $Y$ are locally convex (Hausdorff)
spaces (lcs, in brief) with respective topological duals $X^{\ast}$ and
$Y^{\ast}$. We denote by $w(X^{\ast},X)$, $w(Y^{\ast},Y)$ the corresponding
weak-topologies on $X^{\ast}$ and $Y^{\ast}$. We enlarge $Y$   by adding    the
element $\infty_{Y}$. The extended real line is denoted  $\overline
{\mathbb{R}}:=[-\infty,+\infty]$, and we adopt the convention $+\infty
-(+\infty)=+\infty$.    
Given a family of sets $\{A_{i}\}_{i\in I}$, we denote its union by
brackets:
\[
\bigcup\left[  A_{i}:i\in I\right]  :=\bigcup_{i\in I}A_{i}.
\] 

Given a set $A\subset X$, we denote by $\operatorname*{cl}(A)$,
$\operatorname{int}(A)$, $\operatorname*{co}(A)$, $\operatorname{cone}(A)$ the
\emph{closure}, the \emph{interior}, the \emph{convex hull} and the the
\emph{convex cone }generated by $A$. By $0_{X}$ we represent the zero vector
in $X$, similarly for the spaces $Y,X^{\ast}$ and $Y^{\ast}$. For two sets
$A,B\subset X$ and $\lambda\in\mathbb{R}$ we define the following operations:
\[
A+B:=\left\{  a+b,\ a\in A\text{ and }b\in B\right\}  ,\quad\lambda
A:=\{\lambda a:a\in A\},
\]
and
\[
A\ominus B:=\{x\in X:x+B\subseteq A\}.
\]
In the  previous operations we consider the following  
conventions:
\[
A+\emptyset=\emptyset+A=\emptyset\text{, }\operatorname*{co}(\emptyset
)=\operatorname*{cone}(\emptyset)=\emptyset,\text{ and }\quad0\emptyset
=\{0_{X}\}.
\]

Given a set $T$ we represent the \emph{generalized simplex} on $T$ by
$\Delta(T)$, which is the set of all the functions $\alpha:T\to 
[0,1]$ such that $\alpha_{t}\neq0$ only for finitely many $t\in
T$ and $\sum_{t\in T}\alpha_{t}=1;$ for $\alpha\in\Delta(T)$ we denote
$\operatorname{supp}\alpha:=\{t\in T:\alpha_{t}\neq0\}$. We also introduce the
symbols $\Delta_{n}:=\Delta(\{1,2,\ldots,n\})$, $\Delta_{n}^{\epsilon
}:=\epsilon\Delta_{n}.$

If $A$ is convex and $\epsilon\geq0$, we define the $\epsilon$-\emph{normal
	set} to $A$  at $x$ as
\begin{equation}\label{enormal}
	N_{A}^{\epsilon}(x):=\{x^{\ast}\in X^{\ast}:\ \langle x^{\ast},y-x\rangle
	\leq\epsilon\text{ for all }y\in A\},
\end{equation}
if $x\in A,$ and $N_{A}^{\epsilon}(x)=\emptyset$ if $x\notin A$. If
$\epsilon=0$, $N_{A}^{0}(x)\equiv N_{A}(x)$ is the so-called \emph{normal cone
}to $A$\ at $x.$

If $A\subseteq X$,  the \emph{polar set} of $A$  is given by
\[
A^{\circ}:=\{x^{\ast}\in X^{\ast}:\ \langle x^{\ast},y\rangle
\leq 1\text{ for all }y\in A\}.
\]
The \emph{positive
	polar} \emph{cone} of $A$ is given by%

\[ 
A^{+}:=\left\{  x^{\ast}\in X^{\ast}: \langle x^{\ast}, x \rangle\geq0, \text{
	for all } x\in A \right\}  .
\]

Given a mapping $F:X\rightarrow Y\cup\{{\infty_{Y}\}}$, the (\emph{effective}) \emph{domain} of $F$ is given by
$\operatorname*{dom}F:=\{x\in X:\ F(x)\in Y\}$. We say that $F$ is
\emph{proper} if $\operatorname*{dom}F\neq\emptyset$. Given $y^\ast \in Y^\ast$, we define  $\langle y^{\ast},F\rangle : X\to \mathbb{R}\cup\{+\infty\}$ by
\begin{align}\label{SCALARIZATION_F}
	\langle y^{\ast},F\rangle(x):=\left\{
	\begin{array}
		[c]{ll}%
		\langle y^{\ast},F(x)\rangle, & \text{if }x\in\dom F,\\
		+\infty, & \text{if }x\notin\dom F.
	\end{array}
	\right.
\end{align}

We use similar notations for functions $f:X\rightarrow\mathbb{R}\cup
\{+\infty\}$. We represent by $\Gamma_{0}(X)$ the set of all
functions $f:X\rightarrow\mathbb{R}\cup\{+\infty\}$ which are proper, convex
and lower-semicontinuous (lsc, in brief). 

The continuity of functions and mappings will be only considered at points of their domains.

Given the set $A$, the \emph{indicator function} of $A$  is 
\[
\delta_{A}(x):=\left\{
\begin{array}
	[c]{cc}%
	0, & \text{if }x\in A,\\
	+\infty, & \text{if }x\notin A.
\end{array}
\right.
\]


For $\epsilon \geq 0$, the \emph{$\epsilon $-subdifferential} (or 
\emph{approximate subdifferential}) of a function $f:X\rightarrow \overline{%
	\mathbb{R}}$ at a point $x\in X$ is the
set 
\begin{equation*}
	\partial _{\epsilon }f(x):=\{x^{\ast }\in X^{\ast }:\ \langle x^{\ast
	},y-x\rangle \leq f(y)-f(x)+\epsilon ,\ \forall y\in \mathbb{R}^{n}\};
\end{equation*}%
if $|f(x)|=+\infty $, we set $\partial _{\epsilon }f(x)=\emptyset $. The
special case $\epsilon =0$ yields the classical (\emph{Moreau-Rockafellar}) \emph{convex subdifferential}, denoted by $\partial f(x)$.

%


We finish this subsection by recalling the
following alternative formulation of a general constrained optimization
problem which uses a maximum function.  Since the proof follows standard argument, we omit its   proof. 
\begin{lemma}\label{lemmaproblemlema}
	Given the functions  $g, h: X \to \overline{\mathbb{R}}$  and the nonempty set  $C \subseteq X$, let us consider the optimization problem
	\begin{equation}\label{apendixoptproblemlema}
		\begin{array}{cl}
			&\min g(x)\\
			&\textnormal{s.t. }h(x) \leq 0\\
			&x\in C.
		\end{array}
	\end{equation} 
	Assume that the optimal value, $\alpha$,  of problem \eqref{apendixoptproblemlema}  is
	finite. Then, $\bar{x}$ is an optimal solution of  \eqref{apendixoptproblemlema} if and only
	if $\bar{x}$ is an optimal  solution of the optimization problem
	\begin{equation}\label{apendixoptproblemlema002}
		\begin{array}{l}
			\min    \max\{  g(x)-\alpha, h(x)   \}\\
			\textnormal{s.t. }	\hspace{0.1cm} x\in C.
		\end{array}
	\end{equation}
	Moreover, the optimal value of  problem \eqref{apendixoptproblemlema002} is zero.
\end{lemma}
\begin{remark}
	The function   $H: X\times X\to \Rex$ defined by 
	\[ H(x,y):=\max\{ g(x) - g(y), h(x) \}\] 
	is called standard \emph{improvement function} (see, e.g.,\cite{MR3290054}). Particularly,  the objective function used in problem \eqref{apendixoptproblemlema002} corresponds to the    improvement function  at $y=\bar{x}$. 
\end{remark} 
\subsection{B-DC functions and basic properties}
Next we introduce a new class of DC functions which constitutes the keystone of this paper. It extends the notion of \emph{DC  vector-valued mappings} introduced in \cite{CorreaLopezPerez2021} and is also related to the concept of \emph{delta-convex functions} in \cite{MR1016045}. 
\begin{definition}\label{definition_DCfunction}
	Let $X$ and $Y$ be lcs spaces and $h\in \Gamma_{0}(X)$.
	\begin{enumerate}[label=\roman*)]
		\item\label{DefiDCB} 	Consider a nonempty set $B \subseteq Y^\ast$. We define the set of   \emph{$B$-DC mappings with control $h$}, denoted by $\Gamma_{h}(X,Y,B)$, as the set of all mappings $F: X\to Y \cup\{ \infty_Y\}$ such that  $\operatorname{dom} h \supseteq \operatorname{dom} F$ and  \[\langle \lambda^\ast,F \rangle +h \in \Gamma_{0}(X) \text{    for all } \lambda^\ast\in B.\]  We also say  that $F$ is a \emph{DC mapping with control function $h$ relative to $B$}, or that \emph{$F$ is controlled by $h$  relatively  to $B$}.
		\item We represent  by $\Gamma_h(X)$ the set of all functions  $f: X\to \mathbb{R}\cup\{ +\infty \}$ such that $\operatorname{dom} f \subseteq \operatorname{dom} h$ and  $f + h\in \Gamma_{0}(X)$.
	\end{enumerate}
\end{definition}

\begin{remark}\label{RemarkDefinitionDC}
	It is worth mentioning that Definition \ref{definition_DCfunction}    corresponds to a natural extension of the notion used in \cite{MR1016045}, with the name   delta-convex functions.
	More precisely,  following  \cite{MR1016045}, for  a continuous mapping $F: U \subseteq X\to Y$, where $X,Y$ are Banach spaces and $U$ is an open convex set, we say:
	\begin{enumerate}[label=\alph*)]
		\item   $F$ is a \emph{weak-DC mapping} if for every $\lambda^\ast \in Y^\ast$ there exists $h_{\lambda^\ast}: U \to \mathbb{R}$ convex and continuous such that  $ \langle \lambda^\ast,F \rangle +h_{\lambda^\ast} $ is convex and continuous on $U$.
		\item $F$ is a  \emph{DC mapping} if there exists $h: U \to \mathbb{R}$ convex and continuous such that   $ \langle \lambda^\ast,F \rangle +h $  is convex and continuous for all  $\lambda^\ast \in \mathbb{B}_{Y^\ast}$, where $\mathbb{B}_{Y^\ast}$ is the closed unit ball in $Y^\ast$. In this case it is said that $F$ is a DC mapping with control function $h$.
	\end{enumerate}
	
	Moreover, by \cite[Corollary 1.8]{MR1016045} both definitions are equivalent when $Y$ is  finite dimensional.	In  \cite{MR1016045}, the  focus  is on analytic properties of   vector-valued mappings  defined on convex open sets. Here, according to the tradition in optimization theory, we deal with mappings which admit  \emph{extended values}. It is important to mention that  Definition  \ref{definition_DCfunction} \ref{DefiDCB} reduces to the concept of DC mapping introduced in \cite{MR1016045}, when $B$    is the unit ball in the dual space of $Y$, in the normed space setting. Moreover, if a real-valued function $f$ is DC mapping with control $h$,  then necessarily $-f$ is a DC with control $h$   (see Proposition \ref{PropoBasicPro} \ref{PropoBasicProe} below for more details).

\end{remark}

The next proposition gathers some elementary properties of the class $\Gamma_{h}(X,Y,B)$.
\begin{proposition}[{Basic properties of B-DC mappings}]\label{PropoBasicPro}
	\begin{enumerate}[label=\alph*)]
		\item Let $F\in  \Gamma_{h}(X,Y,B)$, then $\operatorname{dom} F$ is convex.
		
		\item Let $F_i\in \Gamma_{h_i}(X,Y,B)$ for $i=1,\ldots p$ with $\cap_{i=1}^p \dom F_i\neq \emptyset$, then $\sum_{i=1}^p F_i \in \Gamma_{h}(X,Y,B)$, where $h= \sum_{i=1}^p h_i$.
		\item\label{PropoBasicProe}   Let $F \in \Gamma_{h}(X,Y,B)$ with $\operatorname{dom} F = X$ and suppose that $B$ is symmetric, then $-F\in \Gamma_{h}  (X,Y,B)$.
	\end{enumerate}
	
\end{proposition}

\begin{proof}
	\emph{a)} 	It follows from the fact that $\operatorname{dom} F = \operatorname{dom} \left( \langle \lambda^\ast, F\rangle + h  \right)$ for every $\lambda^\ast \in B$. 
	\emph{b)} Let $F:= \sum_{i=1}^p F_i $ and $\lambda^\ast \in B$. Then, for all $x\in X$ we have that 
	$\langle \lambda^\ast, F\rangle + h= \sum_{i=1}^p \left(  \langle\lambda^\ast, F_i\rangle + h_i\right),$	
	which is a convex proper and lsc function. 
	\emph{c)}  It follows from the fact that $ \langle \lambda^\ast , -F \rangle + h= \langle -\lambda^\ast , F \rangle + h$ and $-\lambda^\ast \in B$ due to the symmetry of $B$. 
	
\end{proof}

\subsection{Generalized differentiation  and calculus rules}
In this subsection, we introduce the necessary notation to distinguish nonsmooth and nonconvex functions and mappings, and we  develop some calculus rules. 

The following notions are based on classical \emph{bornological} constructions in Banach spaces  (see, e.g., \cite{MR1664320,MR1384247,MR1669779,MR2191744} for more details and similar constructions). 
Given a locally convex space $X$, we consider $\beta(X)$ as the family  of all bounded sets of $X$ (i.e., those sets such that every  seminorm which generates the topology on $X$ is bounded on them). We  simply write $\beta$, when there is no ambiguity of the space.
\begin{definition}
	We say that $g:X\rightarrow\overline
	{\mathbb{R}}$ is  \emph{differentiable at} $x$, with $|g(x)|<+\infty$,
	if there exists $x^{\ast}\in X^{\ast}$ such that
	
	\begin{equation}\label{defibetadifferentiable}
		\lim_{t\rightarrow 0^+}\sup_{h\in S}\left|  \frac{g(x+th)-g(x)}{t}-\left\langle
		x^{\ast},h\right\rangle \right|  =0 ,\text{ for all } S\in\beta .
	\end{equation}
	
\end{definition}
It is not difficult to see that when such $x^\ast$ exists, it is unique. In that case, and following the usual notation, we simply write $\nabla
g(x)=x^{\ast}$.   Here it is important to recall that in  a general locally convex space $X$, the differentiability of a function does not imply its continuity.  For instance, the square of the norm in any infinite dimensional Hilbert space is Fr\'echet differentiable, but not weak continuous  (see, e.g., \cite{MR0344032}). 
\begin{definition}
	The \emph{regular (Fr\'echet) subdifferential} of $f:X\rightarrow
	\overline{\mathbb{R}}$ at $x\in X,$ with $|f(x)|<+\infty$, is the set,
	denoted by $\hat{\partial} f(x),$ of all $x^{\ast}$ such that 
	
	\begin{equation*}
		\liminf\limits_{t\rightarrow 0^+ } \sup\limits_{h\in S}\left( \frac{f(x+th)-f(x)  }{t}-\langle x^{\ast
		},h\rangle\right)\geq 0,\; \text{ for all } S\in \beta.
	\end{equation*}%
	For a point $x\in X$, where
	$|f(x)|=+\infty$ we simply define $\hat{\partial}  f(x)=\emptyset$.
\end{definition}

Now, let us formally prove that the regular subdifferential coincides with the classic convex subdifferential for functions in $\Gamma_0(X)$.
\begin{lemma}\label{sublemma01}
	Let $f\in \Gamma_0(X)$. Then,  the regular  subdifferential of $f$ coincides with the classic convex subdifferential of $f$, that is,  $ \hat{\partial}   f(x)=\partial f(x) $ for every $x\in X$.
\end{lemma} 

\begin{proof}
	Since $\partial f(x) \subseteq \hat{\partial}   f(x) $ obviously  holds, we focus on the opposite inclusion. Let $x^\ast \in \hat{\partial}   f(x)$, which implies that $|f(x)| <+\infty$. Now, consider $y\in X$ and $\epsilon >0$ arbitrary, and let $S \in \beta$ such that $h= y-x \in S$. Hence,  we have that for small enough $t\in(0,1)$, the following inequality holds
	\[ \frac{ f((1-t)x + ty) - f(x)}{t}-\langle x^\ast , y -x\rangle =\frac{f(x + t(y-x))  - f(x)}{t} -\langle x^\ast , y -x\rangle  \geq-  \epsilon, \]
	so using the convexity of $f$, we yield that $ f(y)-f(x) -\langle x^\ast , y -x\rangle   \geq-\epsilon ,$
	then, taking $\epsilon \to 0$, we have that 
	$f(y) - f(x) \geq \langle x^\ast, y-x\rangle$, which from the arbitrariness of $y\in X$ implies the result.
\end{proof}

The following \emph{sum rule} is applied in the paper, and we provide its proof for completeness.

\begin{lemma}\label{lemma:sumrule}
	Let  $x\in X$, and  $g: X \to \Rex$  be   differentiable at $x$. Then, for any function    $f: X\to \overline{\mathbb{R}}$  we have
	\begin{align}  \label{sumerulefrechet}
		\hat{\partial} ( f + g)(x)= \hat{\partial}  f(x) + \nabla  g(x).
	\end{align}
\end{lemma}

\begin{proof}
	Let us suppose that $x^\ast \in \hat{\partial} ( f + g)(x)$ and  $S\in \beta$. Hence, 
	\begin{align*}
		0\leq  &	\liminf\limits_{t\rightarrow 0^+ } \sup\limits_{h\in S}\left( \frac{f(x+th)+g(x+th)-f(x)-g(x)  }{t}-\langle x^{\ast
		},h\rangle\right) \\
		\leq  & 	\liminf\limits_{t\rightarrow 0^+ } \sup\limits_{h\in S}\left( \frac{f(x+th) -f(x)  }{t}-\langle x^{\ast
		}-\nabla g(x),h\rangle\right)  \\& + \lim\limits_{t\rightarrow 0^+ } \sup\limits_{h\in S}\left( \frac{ g(x+th)-  g(x)  }{t}-\langle  \nabla g(x),h\rangle\right) \\
		=& 	\liminf\limits_{t\rightarrow 0^+ } \sup\limits_{h\in S}\left( \frac{f(x+th) -f(x)  }{t}-\langle x^{\ast
		}-\nabla g(x),h\rangle\right), 
	\end{align*}
	which shows that $x^\ast - \nabla  g(x)  \in \hat{\partial}  f(x)$. To prove the converse inclusion it is enough to notice that
	$\hat{\partial}  f(x) = \hat{\partial}  (f +g - g)(x) \subseteq 	\hat{\partial} ( f + g)(x) -   \nabla  g(x),$ where the final inclusion follows from the previous part, and that  ends the proof.
\end{proof} 
Next, we employ	 the regular subdifferential to provide machinery to  differentiate nonsmooth  vector-valued mappings. 
\begin{definition}
	Given a mapping $F:X\rightarrow Y\cup\{+\infty_{Y}\}$ and $x\in\dom F$ we define
	the  \emph{regular coderivative of} $F$ \emph{at} $x$ by the set-valued map
	$\hat{D}^{\ast}F(x):Y^{\ast}\rightrightarrows X^{\ast}$ defined as:
	\begin{equation}
		\hat{D}^{\ast }F(x)(y^{\ast }):=\hat{\partial}\left(
		\langle y^{\ast },F\rangle \right) (x),  \label{coderivative:def}
	\end{equation}%
	where $\langle y^{\ast},F\rangle$ is the function defined in \eqref{SCALARIZATION_F}.
\end{definition}

This operator is positively homogeneous. In particular, definition
\eqref{coderivative:def} coincides with the general construction for the regular
coderivative of set-valued mappings  on Banach spaces  when
$F$ is calm at $x$, that is,
\[
\left\Vert F(u)-F(x)\right\Vert \leq\ell\left\Vert u-x\right\Vert ,
\]
for some $\ell>0$ and $u$ close to $x$ (see, e.g., \cite[Proposition 1.32]{MR1995438}).

The following lemma yields  a sum rule  for functions in $\Gamma_{h}(X)$.
\begin{lemma}\label{sumerulefrechet:2}
	Let $f_1,f_2 \in \Gamma_{h}(X)$ for some function $h \in \Gamma_0(X)$. Suppose that there exists a point in $\operatorname{dom} (f_2+h)$ where $f_1+h$  is  continuous. Then,
	\begin{align*}
		\hat{\partial} ( f_1+ f_2)(\bar  x) = 	\hat{\partial}   f_1(\bar x ) + 	\hat{\partial}  f_2(\bar x),
	\end{align*}
	provided that $h$ is differentiable at $\bar x$.
\end{lemma}
\begin{proof}
	Let us compute the subdifferential $\partial ( f_1+ f_2+2h)(\bar  x)$. On the one hand,  by Lemma \ref{lemma:sumrule}, we have that $\partial ( f_1+ f_2+2h)(\bar  x) =  \hat\partial   ( f_1+ f_2 )(\bar  x) + 2 \nabla  h(\bar x)$. On the other hand, since the convex function  $f_1 + h$ is continuous at some point in  the domain of  $f_2 + h$, we get (see, e.g., \cite[Theorem 2.8.7]{MR1921556}) that 
	\begin{align*}
		\partial ( f_1+ f_2+2h)(\bar  x) & = \partial ( f_1+h)  (\bar x )  + \partial  ( f_2+ h)(\bar  x)\\
		&= \hat\partial     f_1 (\bar  x) + \nabla h(\bar x) +   \hat\partial    f_2(\bar  x) + \nabla h(\bar x),
	\end{align*}
	where in the last equality we used  Lemma \ref{lemma:sumrule} again, and that concludes the proof.
\end{proof}

Next, we present some calculus rules for the subdifferential of an extended real DC function.

\begin{proposition}
	\emph{\cite[Theorem 1]{MR1147242}}\label{Prop_DCcalculus} Let $g,h\in \Gamma
	_{0}(X)$ be such that both are finite at $x$. Then, for every $%
	\epsilon \geq 0$ 
	\begin{equation*}
		\partial _{\epsilon }\left( g-h\right) (x)=\bigcap\limits_{\eta \geq
			0}\left( \partial _{\eta +\epsilon }g(x)\ominus \partial _{\eta }h(x)\right).
	\end{equation*}
{ Particularly, $g-h$ attains a global minimum at $x$ if and only if
\begin{align*}
	\partial _{\eta }h( x) \subseteq   \partial _{\eta }g( x), \text{ for all } \eta \geq 0.
\end{align*}  }
\end{proposition}

The following result characterizes the $\epsilon $-subdifferential
of the supremum function of an arbitrary family of functions.

\begin{proposition}	\emph{\cite[Proposition 3.1]{MR3907966}}\label{Epsilonformula}
	Let $\{f_t : t \in T \} \subseteq \Gamma_{0}(X)$  and define $f =\sup_{T}  f_t$. Then,  for every  $\epsilon \geq 0$,
	\begin{equation}\label{FORMULABASIS2}
		{\small \hspace{-0.1cm}\partial _{\varepsilon }f(x)\hspace{-0.1cm}=}%
		\bigcap\limits_{\gamma >\varepsilon }\operatorname{cl}^\ast\left\{ \sum\limits_{t\in 
			\operatorname{supp}\alpha }\alpha _{t}\partial _{\eta _{t}}f_{t}(x):\hspace*{-0.1cm}%
		\hspace*{-0.1cm}%
		\begin{array}{c}
			\alpha \in \Delta (T),\;\eta _{t}\geq 0,\;\sum\limits_{t\in \operatorname{supp}%
				\alpha }\alpha _{t}\eta _{t}\in \lbrack 0,\gamma ), \\ \vspace{-0.2cm}  \\
			\sum\limits_{t\in \operatorname{supp}\alpha }\alpha _{t}f_{t}(x)\geq
			f(x)+\sum\limits_{t\in \operatorname{supp}\alpha }\alpha _{t}\eta _{t}-\gamma 
		\end{array}%
		\right\} {\small ,}
	\end{equation}%
	where $\operatorname{cl}^\ast$ represents the closure with respect  to the $w^{\ast }$-topology.

\end{proposition}

The following calculus rules play a key role in our analysis. Given
a mapping $F:X\rightarrow Y\cup \{\infty _{Y}\}$, a set $C\subset Y^{\ast }$
and $\varepsilon \geq 0$ we denote the \emph{$\epsilon $-active index set}
at $x\in \dom F$ by 
\begin{align*}
	C_\epsilon(x):=\left\{ \lambda^\ast \in C: \sup\limits_{\nu^\ast \in C} \langle \nu^\ast,F \rangle (x)  \leq \langle \lambda^\ast , F\rangle(x) + \epsilon \right\}.
\end{align*}
For $\epsilon =0$, we simply write $ C(x):= C_0(x)$.

\begin{theorem}\label{Theo:aprxsubd}
	Let $F:X \to Y\cup \{ \infty_{Y}\}$ and consider $C$ a convex and compact subset of $Y^\ast$ with respect to the $w^\ast$-topology. Let $g\in \Gamma_{0}(X)$ such that for all $\lambda^\ast \in C$ the function $\langle \lambda^\ast , F\rangle + g \in \Gamma_{0}(X)$. Then, for every $\epsilon \geq 0$ and all $x\in X$, we have 
	\begin{align}\label{eq01}
		\partial_\epsilon  \left( \sup\limits_{\lambda^\ast \in C} \langle \lambda^\ast , F\rangle + g  \right)   (x)=\bigcup \left[    \partial_\eta \left( \langle \lambda^\ast , F\rangle + g \right) (x)  	: \begin{array}{c}  \eta \in  [0,\epsilon]
			\text{ and } \\ \lambda^\ast \in C_{\epsilon-\eta}(x) 
		\end{array} 	\right].
	\end{align} 
\end{theorem}
\begin{proof}
	First let us show the inclusion $\supseteq $ in \eqref{eq01}. Consider $x^\ast $ in the right-hand side of \eqref{eq01}; then there exists $\eta \in [0,\epsilon]$ and $\lambda^\ast \in C_{\epsilon-\eta}(x) $ such that $x^\ast \in \partial_\eta \left( \langle \lambda^\ast , F\rangle + g \right) (x) $. Hence, for all $y\in X$
	\begin{align*}
		\langle x^\ast , y - x\rangle& \leq  \langle \lambda^\ast , F\rangle(y) + g(y)  - \langle \lambda^\ast , F\rangle(x) -  g(x)  +\eta\\
		& \leq \sup\limits_{\nu^\ast \in C}   \langle \nu ^\ast , F\rangle(y) + g(y) -  \sup\limits_{\nu^\ast \in C} \langle \nu^\ast,F \rangle (x) -g(x)  +\epsilon - \eta +\eta \\
		&\leq \sup\limits_{\nu^\ast \in C} \langle \nu^\ast , F\rangle(y) + g(y) -  \sup\limits_{\nu^\ast \in C} \langle \nu^\ast,F \rangle (x)-g(x)  +\epsilon.
	\end{align*}
	
	Second, let us consider $T=C$, and the family of functions {$f_t =\langle t , F\rangle + g     $},  for $t=\lambda^\ast$. Then, by Proposition \ref{Epsilonformula} we have
{	
	\begin{align}\label{FORMULABASIS3}
		\partial_{\epsilon} f(x) \subseteq \hspace{-0.05cm} \bigcap\limits_{\gamma >   \epsilon } \cl^\ast\left( \bigcup\left[       \partial_\eta \left( \langle \lambda^\ast , F\rangle + g \right) (x)     :\hspace{-0.12cm} \hspace*{-0.1cm} \begin{array}{c}
			\lambda^\ast\in C, \;   \eta\in [0,\gamma),
			\text{ and }\\\vspace{-0.2cm} \\	\langle \lambda^\ast , F\rangle  (x)    \geq\sup\limits_{\nu^\ast \in C} \langle \nu^\ast,F \rangle (x)  + \eta - \gamma 	
		\end{array}			\right] \right),
	\end{align}}
	where we have simplified \eqref{FORMULABASIS2} by  convexity of $C$ and   linearity of  the application  $\lambda^\ast \mapsto \langle \lambda^\ast , F\rangle  (w) $ for all $w\in X$. In  fact,   for $(\alpha_t)_{t\in T} \in \Delta(T)$  we have
	\begin{align*}
		\sum\limits_{t\in \supp\alpha} \alpha_t  f_t  (\cdot)&= \left( \langle \lambda^\ast , F\rangle  + g \right) (\cdot) & \text{ and  } &
		\sum\limits_{ t\in \operatorname{supp} \alpha} \alpha_t \partial_{\eta_t }  f_t(x) \subseteq  \partial_\eta \left( \langle \lambda^\ast , F\rangle + g \right) (x),
	\end{align*}
	where $\lambda^\ast= \sum_{ t\in \operatorname{supp} \alpha} \alpha_t  t\in C$ and $\eta := \sum_{t\in T}\alpha_t  \eta_t$. Now, consider $x^\ast \in 	\partial_{\epsilon} f(x)$, so that by  \eqref{FORMULABASIS3}, there exists a net $\gamma_\ell \to \epsilon$ and $\eta_\ell \in [0,\gamma_\ell)$, $\lambda^\ast_\ell\in C$ with
	\[ 	\langle \lambda^\ast_\ell , F\rangle  (x)    \geq\sup\limits_{\nu^\ast \in C} \langle \nu^\ast,F \rangle (x)  + \eta_\ell - \gamma_\ell	 \]
	and ${x_\ell}^\ast \in  \partial_{\eta_\ell} \left( \langle \lambda_\ell^\ast , F\rangle + g \right) (x)   $ such that $x^\ast=\lim x_\ell^\ast$. Hence, by compactness of $C$, we can assume that $\lambda_\ell^\ast \to \lambda^\ast\in C$, as well as that $\eta_\ell \to \eta \in [0, \epsilon]$. Furthermore, taking limits 
	$\langle \lambda^\ast , F\rangle  (x)    \geq\sup_{\nu^\ast \in C} \langle \nu^\ast,F \rangle (x)  + \eta - \epsilon	$, so $\lambda^\ast \in C_{\epsilon-\eta}(x) $. Finally, let us show that  $x^\ast \in \partial_\eta \left( \langle \lambda^\ast , F\rangle + g \right) (x) $. Indeed, for all $y\in X$
	\begin{align*} 
		\langle x^\ast_\ell , y- x\rangle \leq \langle \lambda^\ast_\ell , F(y) \rangle + g(y) - \langle \lambda^\ast_\ell , F(x) \rangle + g(x) +\eta_\ell, \text{ for all } \ell \in D,
	\end{align*}
	so taking limits in $\ell$ we conclude that
	\begin{align*}
		\langle x^\ast , y- x\rangle \leq \langle \lambda^\ast , F(y) \rangle + g(y) - \langle \lambda^\ast, F(x) \rangle + g(x) +\eta,
	\end{align*}
	from which, and from the arbitrariness of $y\in X$,  we conclude that $x^\ast$ belongs to $\partial_\eta \left( \langle \lambda^\ast , F\rangle + g \right) (x) $.
\end{proof}

\section{DC mathematical programming}\label{SECT_DCMATHPROGR}

This section is devoted to establishing necessary and sufficient conditions for 
general DC mathematical programming problems. More precisely, we
consider the following optimization problem%
\begin{equation}\label{DC_CONSTRAINT}
	\begin{tabular}{ll}
		$\min $ & $\varphi (x)$ \\ 
		s.t. & $\Phi (x)\in \mathcal{C},$%
	\end{tabular}
\end{equation}%
where $\varphi :X\rightarrow \mathbb{R}\cup \{+\infty \},$ $\Phi
:X\rightarrow Y\cup \{\infty _{Y}\}$ is a vector-valued mapping, and 
$\mathcal{C}\subseteq Y$ is a closed convex set. This section has
two parts devoted to global and local optimality conditions.

\subsection{Global optimality conditions}

Let us establish our main result in this section.

\begin{theorem}\label{THEo:DC01}
	Let $\C \subseteq Y$ be  a closed and convex set such that $0_Y \in \C$ and   $\C^\circ$ is weak$^\ast$-compact.  Let   $\varphi\in \Gamma_{h}(X)$ and  $ \Phi\in \Gamma_{h}(X,Y, \C^\circ)$ for some  control function $h\in \Gamma_0(X)$, and suppose that one of the following conditions holds:
	\begin{enumerate}[ref=\alph*), label=\alph*)]
		\item\label{THEo:DC01:itema} the   function $x \mapsto  \varphi(x) + h(x)$ is continuous at some   point of  $\operatorname{dom} \Phi$, 
		\item\label{THEo:DC01:itemb} the   function $x \mapsto \sup\{  \langle \lambda^\ast , \Phi(x)\rangle  : \lambda^\ast \in   \C^\circ\} +h(x)$ is   continuous at some point of $\operatorname{dom} \varphi$.
	\end{enumerate}
	Then, if  $\bar x$ is  an optimal solution of the optimization problem  \eqref{DC_CONSTRAINT}, we have that
	\begin{align}\label{THEo:DC01:eq00}
		\partial_{\eta}  	h( \bar{x})  \subseteq \bigcup \Big[    \alpha_1 \partial_{\eta_1 	}( \varphi + h ) ( \bar{x})  +  \alpha_2 \partial_{\eta_2} \left( \langle \lambda^\ast , \Phi\rangle + h \right) (\bar{x})  		\Big], \text{ for all } \eta \geq 0,
	\end{align}
	where the union is taken over all $ \eta_1,\eta_2 \geq 0$, $(\alpha_1,\alpha_2) \in \Delta_2$  and  $  \lambda^\ast \in \C^\circ $ such that  
	\begin{align}\label{THEo:DC01:eq00scalar}
		\alpha_1   \eta_1+  \alpha_2 ( \eta_2 +1-  \langle \lambda^\ast , \Phi(\bar {x} ) \rangle ) =  \eta.
	\end{align}
	
	Conversely, assume that $\bar{x}$ is a feasible point of \eqref{DC_CONSTRAINT}  and that \eqref{THEo:DC01:eq00} always holds with  $\alpha_1 >0 $, then $\bar{x}$ is a solution of \eqref{DC_CONSTRAINT}  relative to $\operatorname{dom} \partial h$, that is, $\bar{x}$ is   an optimum  of  $\min\{ \varphi(x): x\in \Phi^{-1}(\mathcal{C})\cap  \operatorname{dom} \partial h\}$. 
	
	
\end{theorem}
\begin{remark}[before the proof]
	It is important to note that the assumption that $C^\circ$ is $w^\ast$-compact is not restrictive. Indeed, whenever there exists some $z_0$ in $\C$ such that $\C_{z_0}^\circ:=\left( \C- z_0\right)^\circ$ is weak$^\ast$-compact,  Theorem \ref{THEo:DC01}  can be translated easily in terms the mapping $\hat{\Phi} (x):=\Phi(x)-z_0$ and the set $\hat{\C}= \C -z_0$.    Moreover, according to Banach-Alaouglu-Bourbaki  theorem, in order to guarantee that $\C^\circ_{z_0}$ is weak$^{\ast}$-compact, it is enough to suppose that $z_0\in\inte (\C)$. More precisely, 	\cite[Theorem I-14]{MR0467310} establishes  that $z_0$ belongs to the interior of $\C$ with respect to the Mackey topology if and only if $\C_{z_0}^\circ$ is weak$^\ast$-compact. Here, it is important to  mention that there are several relations between the weak$^\ast$-compactness of $\C^\circ_{z_0}$ and the nonemptiness of the interior of $\C$, with respect to the Mackey topology (see, e.g.  \cite{MR0467310,MR0467080} for more details), and they can be connected even with the classical James's Theorem (see \cite{MR4014679}) and other variational and geometric properties of functions (see   \cite{MR3507100,MR3767762,MR3509670}). 
\end{remark}
\begin{proof}
	First let us suppose that $\bar{x}$ is a solution of \eqref{DC_CONSTRAINT} and let $\alpha$ be the optimal value of the optimization problem  \eqref{DC_CONSTRAINT}. In the first part we prove two  claims.
	
	\noindent\emph{  Claim 1:} First we prove  that \begin{align}\label{THEo:DC01:eq01}
		\partial _{\eta }h(\bar x) \subseteq   \partial _{\eta }\psi(\bar x), \text{ for all } \eta \geq 0,
	\end{align}
	where 
	\begin{align}\label{deffuncpsi}
		\psi (x)&:=\max\left\{ \psi_1 (x),\psi_2(x)		\right\}, \;\; 	
	\end{align} and 
	\begin{align*}
		\psi_1 (x)  :=\varphi(x) + h(x)-\alpha, \text{ and }	 \psi_2(x):= f(x)  +h(x),
	\end{align*}   
	with \begin{align}\label{defsupf}
		f(x):=\sup\left\{  \langle \lambda^\ast , \Phi (x)  \rangle    : \lambda^\ast \in   \C^\circ \right\} -1.
	\end{align}
	Indeed, first let us notice that, by the bipolar theorem,  
	\begin{eqnarray*}
		\Phi (x) \in \mathcal{C} \Leftrightarrow  \Phi (x)\in (\mathcal{C}^{\circ
		})^{\circ } 
		\Leftrightarrow \langle \lambda ^{\ast },\Phi (x)\rangle \leq 1\text{	for all }\lambda^{\ast }\in \mathcal{C}^{\circ }\Leftrightarrow
		f(x)\leq 0.
	\end{eqnarray*}

	Therefore, by Lemma \ref{lemmaproblemlema}, the  optimization problem \eqref{DC_CONSTRAINT}  has the same optimal solutions as the   problem 
	\begin{align*}
		\min_{x\in X}  \left(  \max\{  \varphi(x) + h(x)-\alpha, f(x)  +h(x)  \} - h(x) \right).
	\end{align*}  
Hence, by  Proposition \ref{Prop_DCcalculus}, we have that $\bar x$ is a solution of \eqref{DC_CONSTRAINT} if and only if  \eqref{THEo:DC01:eq01} holds.

	\noindent\emph{  Claim 2:} Next we prove that  $\partial _{\eta }\psi( \bar{x})$ is precisely the right-hand side of \eqref{THEo:DC01:eq00}, and consequently  \eqref{THEo:DC01:eq00} holds.
	
	Let us compute $\partial _{\eta }\psi( \bar{x})$ using the formula for the $\eta$-subdifferential of the max-function  given in \cite[Corollary  2.8.15]{MR1921556} (the assumptions are actually satisfied thanks to  conditions \ref{THEo:DC01:itema} or \ref{THEo:DC01:itemb}), so we get
	\begin{align*}
		\partial _{\eta }\psi( \bar{x}) = \bigcup\left[  \alpha_1 \partial_{\frac{\epsilon_1}{\alpha_1}  	} \psi_1 ( \bar{x}) + 	\alpha_2 \partial_{\frac{ \epsilon_2}{\alpha_2}}  	 \psi_2( \bar{x})   : \begin{array}{c}
			(\epsilon_0,\epsilon_1, \epsilon_2) \in \Delta_3^\eta, \, (\alpha_1,\alpha_2) \in \Delta_2 \\
			\alpha_1  \psi_1 ( \bar{x}) +  \alpha_2 \psi_2 ( \bar{x}) \geq \psi( \bar{x}) -\epsilon_0
		\end{array}		\right].
	\end{align*}
	Now, relabelling $\eta_1:=  {\epsilon_1}/{\alpha_1}$, $\eta_0:= {\epsilon_2}/{\alpha_2}$  for $\alpha_1 \neq 0\neq \alpha_2$ and  $\eta_1= \eta_0=0$ otherwise,  and using   that $ \psi(\bar x)=\psi_1 (\bar{x})=h(\bar{x})$ and $\psi_2(\bar{x})= f(\bar{x}) + h(\bar{x})\leq h(\bar{x})$ we get
	\begin{align}\label{THEo:DC01:eq02}
		\partial _{\eta }\psi( \bar{x})  
		= \bigcup\left[  \alpha_1 \partial_{\eta_1 	} \psi_1 ( \bar{x}) + 	\alpha_2 \partial_{\eta_0}  	 \psi_2( \bar{x})   : \begin{array}{c}
			\eta_1 ,\eta_0 \geq 0,  \, (\alpha_1,\alpha_2) \in \Delta_2 \\
			\alpha_1   \eta_1+  \alpha_2 ( \eta_0 -  f(\bar x )  ) \leq  \eta 
		\end{array}	\right].
	\end{align}
	Now, we compute  $ \partial_{\eta_0}  	 \psi_2( \bar{x}) $ using  Theorem \ref{Theo:aprxsubd}, so
	\begin{align}\label{THEo:DC01:eq03}
		\partial_{\eta_0}  	 \psi_2( \bar{x})  = \bigcup \left[   \partial_{\eta_2} \left( \langle \lambda^\ast , \Phi\rangle + h \right) (\bar{x})  	: \begin{array}{c}  \eta_2 \in  [0,\eta_0]
			\text{ and } \\ \lambda^\ast \in (\C^\circ)_{\eta_0-\eta_2}(\bar{x}) 
		\end{array} 	\right].
	\end{align}
	Hence, combining \eqref{THEo:DC01:eq02} and \eqref{THEo:DC01:eq03}  we  conclude that  $\partial _{\eta }\psi( \bar{x})$ is given by  the right-hand expression in  \eqref{THEo:DC01:eq00}. {  Moreover, by \eqref{THEo:DC01:eq03}, we have that $f(\bar{x}) \leq  \langle \lambda^\ast , \Phi(\bar {x} ) \rangle  -1 + \eta_0- \eta_2$. Hence,  using \eqref{THEo:DC01:eq02}, we get 
\begin{align*}
	\alpha_1   \eta_1+  \alpha_2 ( \eta_2 +1-  \langle \lambda^\ast , \Phi(\bar {x} ) \rangle ) \leq 	\alpha_1   \eta_1+  \alpha_2 ( \eta_0 - f(\bar{x}) \rangle ) \leq 	\eta
\end{align*}	}
	Finally, by \eqref{THEo:DC01:eq01} and the last computation   we get  that   
	\eqref{THEo:DC01:eq00} holds by increasing the values of $\eta_1$ and $\eta_2$ if it is needed.
	
	Now, to prove the converse, consider $y \in \operatorname{dom} \partial h$ a feasible point of the optimization problem \eqref{DC_CONSTRAINT}. Then, if we consider $x^\ast \in \partial h(y)$, we have that 
	\begin{align}\label{Main01}
		h(y) -h(\bar{x}) +\eta &= \langle x^\ast , y- \bar{x}\rangle, \text{ and }x^\ast \in \partial_\eta h(\bar{x}),
	\end{align}
	where $\eta := h^\ast(x^\ast) + h(\bar{x}) -\langle x^\ast ,  \bar{x}\rangle \geq 0.$ Hence, by \eqref{THEo:DC01:eq00} we see that there are 
	$ \eta_1,\eta_2 \geq 0$, $(\alpha_1,\alpha_2) \in \Delta_2$  and $  \lambda^\ast \in \C^\circ $  satisfying \eqref{THEo:DC01:eq00scalar} with  $\alpha_1\neq 0$, and 
	\begin{align*}
		x^\ast \in \alpha_1 \partial_{\eta_1 	}( \varphi + h ) ( \bar{x})  +  \alpha_2 \partial_{\eta_2} \left( \langle \lambda^\ast , \Phi\rangle + h \right) (\bar{x}).	
	\end{align*}
	Particularly,
	\begin{align}
		\nonumber\langle x^\ast, y -\bar{x} \rangle \leq& \alpha_1 \left( \varphi(y)  + h(y) -\varphi(\bar{x} ) -h(\bar{x})  +\eta_1  \right)  \\
		\nonumber&+  \alpha_2 \left(  \langle \lambda^\ast , \Phi\rangle(y) + h(y) - \langle \lambda^\ast , \Phi\rangle(\bar x) - h (\bar x)  +\eta_2 \right) \\
		\nonumber \leq & \alpha_1 \left( \varphi(y)  - \varphi(\bar{x}) \right) + \alpha_1 \eta_1 +    \alpha_2 \left( \eta_2 + 1 -	\langle \lambda^\ast , \Phi(\bar x)	\rangle  \right)    + 		 h(y)-h(\bar{x})\\
		\leq &\alpha_1 \left( \varphi(y)  - \varphi(\bar{x}) \right)   + h(y)-h(\bar{x}) +\eta.\label{Main02}
	\end{align}
	We conclude, using \eqref{Main01} and \eqref{Main02}, that 
	\begin{align}\label{Main03}
		0 \leq \alpha_1( \varphi(y)  - \varphi(\bar{x})),
	\end{align}
	which, recalling that $\alpha_1 > 0$, ends the proof.
\end{proof}

\begin{remark}
	Let us briefly   comment	on  some facts about  the last result:
	\begin{enumerate}[label=\roman*)]
		\item Notice that  $  \lambda^\ast \in \C^\circ $  if and only if $\lambda^\ast \in N^{\eta_3}_{\C}(\Phi(\bar{x}))$, where $\eta_3:=1 - \langle \lambda^\ast ,  \Phi(\bar{x})\rangle \geq 0$. Indeed, since $\Phi(\bar{x}) \in \C$, we have
		\begin{align*}
			\lambda^\ast \in \C^\circ &\Leftrightarrow \langle \lambda^\ast , u\rangle \leq  1, \; \text{  for all } u \in \C\\
			&  \Leftrightarrow \langle \lambda^\ast , u- \Phi(\bar{x})\rangle \leq1- \langle \lambda^\ast,  \Phi(\bar{x})\rangle \text{  for all } u \in \C\\
			& \Leftrightarrow \lambda^\ast \in N^{\eta_3}_{\C}(\Phi(\bar{x})) \text{ with } \eta_3:=1 - \langle \lambda^\ast ,  \Phi(\bar{x})\rangle \geq 0.
		\end{align*}
		Therefore, the existence of multipliers can be equivalently described in terms of the $\epsilon$-normal set defined in \eqref{enormal}. Nevertheless, we  prefer to use the set $\C^\circ$ in order to take advantage of the compactness of this set, which will be exploited later.
		
		\item 	The converse in  Theorem \ref{THEo:DC01}  can be proved assuming that \eqref{THEo:DC01:eq00} always holds with  multiplier $\alpha_1 \neq 0$  for all $\eta \in  [0,  \bar{\eta}]$, where $$\bar \eta:= \sup \left\{  h^\ast(y^\ast) + h(  \bar{x}) - \langle y^\ast , \bar{x} \rangle	: 		y^\ast \in \partial h(y),  \; y \in\Phi^{-1}(\C) 	\right\}.$$
		\item It is worth mentioning that, due to the assumption of the existence of a continuity point of $\varphi +h$, we can prove that $\varphi,h$ are bounded above on a neighbourhood of a point of their domain. Indeed, let us suppose that $\varphi +h$ is bounded above on a neighbourhood $U$ of $x_0$, that is, $\varphi (x) + h(x) \leq M$ for all $x\in U$  for some scalar $M$.  Since $\varphi$ and $h$ are lsc at $x_0$, we can assume (shrinking enough the neighbourhood $U$) that $\inf_{x\in U}\varphi(x)>-m $ and $\inf_{x\in U} h(x) >-m$, for some constant $m\in \mathbb{R}$, which implies that $\varphi (x) \leq  M+m$ and $ h(x) \leq M +m$ for all $x\in U$. Hence,  due to the convexity of the involved functions, we see that  $\varphi$ and  $h$ are continuous  on $ \inte \operatorname{dom} \varphi$ and $\inte \operatorname{dom} h$, respectively (see, e.g., \cite[Theorem 2.2.9]{MR1921556}). 	Particularly, the latter implies that  $\operatorname{dom} \partial h \supseteq \inte \operatorname{dom} h$ is dense on $\operatorname{dom} \varphi$ (recall $\dom \varphi \subset \dom h$).
	\end{enumerate}
	
\end{remark}

Now, let us establish the following corollary when the problem \eqref{DC_CONSTRAINT} is convex. The proof follows directly from Theorem \ref{THEo:DC01}, so we omit the details.

\begin{corollary}\label{THEo:DC01_conv}
	Let $\C \subseteq Y$ be  a closed and convex set such that  $0_Y \in \C$ and   $\C^\circ$ is weak$^\ast$-compact.  Let   $\varphi\in \Gamma_{0}(X)$ and  $ \Phi\in \Gamma_{0}(X,Y, \C^\circ)$ and suppose that one of the following conditions holds:
	\begin{enumerate}[ref=\alph*), label=\alph*)]
		\item\label{THEo:DC01:itema_conv} the   function $\varphi$ is continuous at some   point of  $\operatorname{dom} \Phi$, 
		\item\label{THEo:DC01:itemb__conv} the   function $x \mapsto \sup\{  \langle \lambda^\ast , \Phi(x)\rangle  : \lambda^\ast \in   \C^\circ\}$ is   continuous at some point of $\operatorname{dom} \varphi$.
	\end{enumerate}
	Then, if  $\bar x$ is  an optimal solution of the optimization problem  \eqref{DC_CONSTRAINT}, we have that there exists $(\alpha_1,\alpha_2) \in \Delta_2$  and  $  \lambda^\ast \in \C^\circ $ such that 
	
	\begin{align}\label{COR:DC01:eq00_conv}
		0_{X^\ast} \in  \alpha_1 \partial\varphi  ( \bar{x})  +  \alpha_2 \partial \left( \langle \lambda^\ast , \Phi\rangle \right) (\bar{x})  \text{ and }  \alpha_2 ( 1-  \langle \lambda^\ast , \Phi(\bar {x} ) \rangle )	=0. 
	\end{align}

	Conversely, assume that $\bar{x}$ is a feasible point of \eqref{DC_CONSTRAINT}  and that \eqref{COR:DC01:eq00_conv}   holds with  $\alpha_1 >0 $, then $\bar{x}$ is a solution of \eqref{DC_CONSTRAINT}.
	
	
\end{corollary}

The following result shows that the fulfilment  of \eqref{THEo:DC01:eq00} with    $\alpha_1 \geq \epsilon_0$, for some $\epsilon_0 >0$, can be used to establish that $\bar{x}$ is a solution to problem \eqref{DC_CONSTRAINT}.

\begin{theorem}
	In the setting of Theorem \ref{THEo:DC01}, suppose that 	 $\bar{x}$ is a feasible point of \eqref{DC_CONSTRAINT}  and that \eqref{THEo:DC01:eq00} always holds with    $\alpha_1 \geq  \epsilon_0 $, for some $\epsilon_0>0$. Then $\bar{x}$ is an optimal solution of \eqref{DC_CONSTRAINT}.
\end{theorem}

\begin{proof}
	The proof follows similar arguments to those given in Theorem
	\ref{THEo:DC01}. The only difference is that, given $y\in\operatorname{dom}h$ and
	$\epsilon^{\prime}\in(0,\epsilon_{0})$, we take $x^{\ast}\in\partial
	_{\epsilon^{\prime}}h(y)$ (recall that $\partial_{\epsilon^{\prime}}h(y)$ is
	nonempty due to the lsc of the function $h$). Again we consider
	$\eta:=h^{\ast}(x^{\ast})+h(\bar{x})-\langle x^{\ast},\bar{x}\rangle\geq0$,
	entailing $x^{\ast}\in\partial_{\eta}h(\bar{x})$. Therefore,
	\eqref{THEo:DC01:eq00} gives rise to the existence of $\eta_{1},\eta_{2}\geq
	0$, $(\alpha_{1},\alpha_{2})\in\Delta_{2}$ and $\lambda^{\ast}\in
	\mathcal{C}^{\circ}$ satisfying \eqref{THEo:DC01:eq00scalar} with
	$\alpha_{1}\geq\epsilon_{0}$, and
	\[
	x^{\ast}\in\alpha_{1}\partial_{\eta_{1}}(\varphi+h)(\bar{x})+\alpha
	_{2}\partial_{\eta_{2}}\left(  \langle\lambda^{\ast},\Phi\rangle+h\right)
	(\bar{x}),
	\]
	which itself yields
	\[
	\langle x^{\ast},y-\bar{x}\rangle\leq\alpha_{1}\left(  \varphi(y)-\varphi
	(\bar{x})\right)  +h(y)-h(\bar{x})+\eta.
	\]
	Then, using that $-h^\ast(x^\ast)+ \langle x^{\ast},y\rangle= \langle x^{\ast},y-\bar{x} \rangle +   h(\bar{x}) - \eta  $, we write
	\[
	\langle x^{\ast},y\rangle -h^\ast(x^\ast)-h(y)\leq \langle x^{\ast},y-\bar{x} \rangle  +h(\bar{x}) - \eta - h(y)  \leq  \alpha_{1}\left(  \varphi(y)-\varphi
	(\bar{x})\right).
	\]
	But now $x^{\ast}\in\partial_{\epsilon^{\prime}}h(y)$ leads us to to the
	inequality $-\epsilon^{\prime}\leq\alpha_{1}(\varphi(y)-\varphi(\bar{x}))$
	instead of \eqref{Main03}, hence $-\frac{\epsilon^{\prime}}{\epsilon_{0}}%
	\leq\varphi(y)-\varphi(\bar{x})$ and taking $\epsilon^{\prime}\rightarrow
	0^{+}$, we get the aimed conclusion.
\end{proof} 

\begin{corollary}\label{corollary:constraint}
	Let $\Constr \subseteq X$ and $\C\subseteq Y$ be  closed and convex set with $0_Y \in \C$ and  suppose     $\C^\circ$ is weak$^\ast$-compact. Let   $\varphi\in \Gamma_{h}(X)$ and  $ \Phi\in \Gamma_{h}(X,Y,\C^\circ)$ for some  function $h\in \Gamma_0(X)$. Consider the following optimization problem
	
	\begin{equation}\label{Main_Problem}
		\begin{aligned}
			& \min \varphi (x)      \\
			&	\textnormal{s.t. } \Phi(x)\in \C,\;  \\ &\hspace{0.6cm}x \in \Constr.
		\end{aligned}
	\end{equation}%
	
	Additionally, assume  that there exists a  point in $\Constr$ such that $\varphi$, $\Phi$ and $ h $  are continuous at this point. Then, if $\bar x$ is an optimal solution of the optimization problem  \eqref{Main_Problem} we have that 
	\begin{align}\label{Cor:DC01:eq00}
		\partial_{\eta}  	h( \bar{x})  \subseteq \bigcup \Big[    \alpha_1 \partial_{\eta_1 	}( \varphi + h ) ( \bar{x})   +  \alpha_2 \partial_{\eta_2} \left( \langle \lambda^\ast , \Phi\rangle + h \right) (\bar{x}) +  N_{\Constr}^{\eta_3} (\bar{x})		\Big] \text{ for all } \eta \geq 0,
	\end{align}
	where the union is taken over all {$ \eta_1,\eta_2,\eta_3 \geq 0$}, $(\alpha_1,\alpha_2) \in \Delta_2$  and $  \lambda^\ast \in \C^\circ $ such that  
	\begin{align*}
		\alpha_1   \eta_1+  \alpha_2 ( \eta_2 +1-  \langle \lambda^\ast , \Phi(\bar {x} ) \rangle ) +   \eta_3=  \eta.
	\end{align*}
	
	Conversely, assume that $\bar{x}$ is a feasible point of \eqref{Main_Problem}  and that \eqref{Cor:DC01:eq00} always holds with    $\alpha_1 >0 $, then $\bar{x}$ is a solution of \eqref{Main_Problem}  relative to $\operatorname{dom} \partial h$. 
	
\end{corollary}
\begin{proof}
	Let us observe that the optimization problem \eqref{Main_Problem} is equivalent to 
	
	\begin{equation}\label{Main_Problem01}
		\begin{aligned}
			\min \varphi (x) 	\textnormal{ s.t. } \Phi_\Constr(x)\in \C,
		\end{aligned}
	\end{equation}%
	where 
	\begin{align}\label{DEFPHIQ}
		\Phi_\Constr(x) :=\left\{  \begin{array}{cc}
			\Phi(x),& \text{ if } x\in \Constr,\\
			\infty_{Y},& \text{ if } x\notin \Constr.
		\end{array}       \right.
	\end{align}
	Furthermore, it is easy to prove that $\langle \lambda^\ast,\Phi_\Constr\rangle = \langle \lambda^\ast , \Phi \rangle + \delta_{\Constr }$ for every $\lambda^\ast \in \C^\circ$, and consequently  $\Phi_\Constr \in \Gamma_{h}(X,Y,\C^\circ)$. Then we apply Theorem  \ref{THEo:DC01} to the optimization problem \eqref{Main_Problem01} (notice that $\varphi +h$ is continuous at some point of $\dom \Phi_{\Constr}$) and  we use 	the sum rule for the $\epsilon $-subdifferential (see, e.g., \cite[Theorem 2.8.3]{MR1921556}) to compute the $\epsilon$-subdifferential of $\langle \lambda^\ast , \Phi \rangle +  h + \delta_{\Constr } $ in terms of the corresponding subdifferentials of $\langle \lambda^\ast , \Phi \rangle +  h$ and $\delta_{\Constr } $ (here recall that $ \langle \lambda^\ast , \Phi \rangle $ and $h$ are continuous at some point of $\Constr$). 
\end{proof} 
\subsection{Local optimality conditions}
In this section we present necessary and sufficient conditions for local optimality of problem \eqref{DC_CONSTRAINT}.  The first result corresponds to a necessary optimality condition. 
\begin{theorem}\label{theo:local:opt}
	In the setting of Theorem \ref{THEo:DC01}, let  $\bar x$ be  a local  optimal solution of the optimization problem  \eqref{DC_CONSTRAINT} and suppose that $h$ is differentiable at $\bar{x}$. Then, we have that   there are  multipliers  $(\alpha_1,\alpha_2)\in\Delta_2$ and  $\lambda^\ast\in  \C^\circ$       such that 
	\begin{align}\label{Theo:DC01:local:eq00}
		0_{X^\ast}  \in   \alpha_1 \hat{\partial} \varphi  ( \bar{x})  + \alpha_2  \hat{D}^\ast \Phi (\bar x)( \lambda^\ast ) \text{ with }\alpha_2(  1 -\langle \lambda^\ast ,  \Phi(\bar {x} )\rangle)=0.
	\end{align}
	In addition, if the following qualification condition holds
	\begin{align}\label{QC01:local}
		0_{X^\ast} \notin \bigcup\left[     \hat{D}^\ast \Phi (\bar x)( \lambda^\ast )  	: \begin{array}{c}
			\lambda^\ast \in  \C^\circ   \text{ such that}\\
			\langle \lambda^\ast , \Phi(\bar {x} ) \rangle   = 1
		\end{array}	\right],
	\end{align}
	then	we have 
	\begin{align}\label{Theo:DC01:local:eq02}
		0_{X^\ast}  \in      \hat{\partial}\varphi  ( \bar{x}) + \operatorname{cone}\left(  \bigcup\left[      \hat{D}^\ast \Phi (\bar x)( \lambda^\ast )  	: \begin{array}{c}
			\lambda^\ast \in  \C^\circ   \text{ such that}\\
			\langle \lambda^\ast , \Phi(\bar {x} ) \rangle   = 1 
		\end{array}	\right] \right).
	\end{align}
	
\end{theorem}
\begin{proof}
	Consider a closed convex  neighbourhood  $U$ of $\bar x$ such that  $\bar x$ is a global optimum for  the next optimization problem
	\begin{equation} \label{Pro01}
		\min \varphi (x) 	\textnormal{ s.t. } \Phi(x)\in \C,\;   x \in U.
	\end{equation}%
	Following the proof of Theorem \ref{THEo:DC01},    we can prove that $\bar x$ is a solution to the unconstrained minimization problem
	\begin{align*}
		\min  \left(  \max\{  \varphi(x) + h(x)-\alpha, f(x)  +h(x)  \}+\delta_{U}(x) - h(x) \right),
	\end{align*}  
	where $\alpha $ is the optimal value of \eqref{Pro01}, and $f$ is defined in \eqref{defsupf}. Now, applying the Fermat rule, and using Proposition \ref{Prop_DCcalculus} (with $\epsilon=\eta=0$), we have that $\bar x$ satisfies the following subdifferential inclusion
	\begin{align*}
		\nabla   	h( \bar{x})  \in   \partial \left(  \psi  + \delta_{U} \right) (\bar{x}), 
	\end{align*}
	where $\psi$ is the function introduced in \eqref{deffuncpsi}. Since $U$ is a neighbourhood of $\bar{x}$, we have that $\partial \left(  \psi  + \delta_{U} \right) (\bar{x}) =\partial   \psi   (\bar{x})$. Moreover, using \emph{Claim 2} of Theorem \ref{THEo:DC01}, we conclude the existence $(\alpha_1,\alpha_2) \in \Delta_2  $ and  $\lambda^\ast \in \C^\circ$ with  $  \alpha_2(  1 -\langle \lambda^\ast ,  \Phi(\bar {x} )\rangle)=0$ such that 
	\begin{align}\label{inlc:01}
		\nabla   	h( \bar{x})  \in   \alpha_1 	\partial( \varphi + h ) ( \bar{x}) + \alpha_2\partial \left( \langle \lambda^\ast , \Phi\rangle + h \right) (\bar{x}).
	\end{align}	 
	Moreover, by    Lemmas \ref{sublemma01} and \ref{lemma:sumrule} (recall that $h$  is differentiable at $\bar{x}$) we can compute the corresponding subdifferentials as  
	\begin{align*}
		\partial( \varphi + h ) ( \bar{x}) &=\hat{\partial}( \varphi + h ) ( \bar{x}) =\hat{\partial} \varphi (\bar x)+ \nabla h(\bar x),\\
		\partial \left( \langle \lambda^\ast , \Phi\rangle + h \right) (\bar{x})&=	 \hat\partial \left( \langle \lambda^\ast , \Phi\rangle + h \right) (\bar{x}) =\hat \partial   \langle \lambda^\ast , \Phi\rangle (\bar{x})  + \nabla h(\bar x)\\&=\hat{D}^\ast \Phi(\bar{x} )(\lambda^\ast )+ \nabla h(\bar x).
	\end{align*}
	Therefore, inclusion \eqref{inlc:01} reduces to \eqref{Theo:DC01:local:eq00}. Now,  \eqref{Theo:DC01:local:eq00} gives us the existence of  $(\alpha_1,\alpha_2) \in \Delta_2  $ and  $\lambda^\ast \in \C^\circ$ with  $ \alpha_2(  1 -\langle \lambda^\ast , \Phi(\bar {x} )\rangle)=0$   such that \eqref{Theo:DC01:local:eq00} holds. 
	Moreover, by the qualification condition \eqref{QC01:local}, we have that $\alpha_1 \neq 0$. Therefore, dividing by $\alpha_1$ we get \eqref{Theo:DC01:local:eq02}.
\end{proof}

\begin{remark}[On normality of multiplier $\lambda^\ast$]
	It is important to emphasize that the conditions $\lambda^\ast \in  \C^\circ$ and $ \langle \lambda^\ast , \Phi(\bar {x} ) \rangle   = 1 $ imply that $\lambda^\ast\in N_{\C}(\Phi(\bar{x}))$. Therefore, in Theorem \ref{theo:local:opt} the multiplier $\lambda^\ast$ is necessarily a normal vector to $\C$ at $\Phi(\bar{x})$. Furthermore, the qualification condition \eqref{QC01:local} is equivalent to 
	\begin{align*} 
		0_{X^\ast} \notin \bigcup\left[     \hat{D}^\ast \Phi (\bar x)( \lambda^\ast )  	: \begin{array}{c}
			\lambda^\ast \in  N_{\C}(\Phi(\bar{x}))   \text{ such that}\\
			\langle \lambda^\ast , \Phi(\bar {x} ) \rangle   = 1
		\end{array}	\right].
	\end{align*}
	Here, the equality $	\langle \lambda^\ast , \Phi(\bar {x} ) \rangle   = 1$ is relevant  because, without this condition, $0_{X^\ast}$   always belongs to the set $\bigcup\left[     \hat{D}^\ast \Phi (\bar x)( \lambda^\ast )  	:  
	\lambda^\ast \in  N_{\C}(\Phi(\bar{x}))   
	\right].$
\end{remark}

\begin{remark}[On abstract  differentiability]
	It is worth mentioning that in Theorem \ref{theo:local:opt} above   the differentiability and subdifferentiability can be exchanged for a more general notion using an abstract concept of  subdifferentiability. Indeed, based on the notion  of presubdifferential  (see, e.g., \cite{MR1312029,MR1357833}), we can adapt the definition there in the following way:	For every $x \in X$ consider a family of functions $\mathcal{F}_x$, which are finite-valued at $x$. 	Now, consider an operator $\asub$ which associates to  any lower semicontinuous function $f: X\to \overline{\mathbb{R}}$ and any $x \in X$, a subset $\asub f(x)$ of $X^\ast$  with the following properties:
	\begin{enumerate}[label=\roman*)]
		\item $\asub f(x) =\emptyset $ for all $x$ where $|f(x)|=+\infty$.
		\item  $\asub f(x) $ is equal to the convex subdifferential whenever $f$ is proper, convex and lower semicontinuous.
		\item $\asub \phi(x)$ is single valued for every $x \in X$ and $\phi \in \mathcal{F}_x$.  In that case $\phi$ is called $\asub$-differentiable at $x$, and we represent by $ \tilde{\nabla} \phi(x)$  the unique point in $\asub \phi(x)$.
		\item For every $x \in \operatorname{dom} f$ and $\phi \in \mathcal{F}_x$, we have 
		$$ \asub \left( f + \phi\right)(x) \subseteq \asub   f(x) +  \tilde{\nabla} \phi(x).$$
	\end{enumerate}
	
	The above notion covers several classes of subdifferentials, for instance: 
	\begin{enumerate}[label=\arabic*)]
		\item Bornological subdifferential (Fr\'echet, Hadamard, Gateaux, etc) with  $\mathcal{F}_x$,   the family of differentiable functions at $x$ with respect to that bornology.
		\item Viscosity bornological subdifferential with  $\mathcal{F}_x$,   the family of  smooth functions (with respect to that bornology) at $x$ .
		\item Proximal subdifferential with  $\mathcal{F}_x$,   the family of $\mathcal{C}^2$-functions  at $x$.
		\item Basic  subdifferential  with $\mathcal{F}_x$, the family of $\mathcal{C}^1$-functions  at $x$.
	\end{enumerate}

	Using the above definition, we can define the notation of $\asub$-coderivative similar to \eqref{coderivative:def}, given by 
	$\tilde{D}^\ast \Phi(x)(y^{\ast }):=\tilde{\partial}\left( \langle y^{\ast },\Phi\rangle \right) (x).$ Using these tools, it is easy to change the proof of Theorem \ref{theo:local:opt} requiring that the convex function $h$  belongs to $ \mathcal{F}_{\bar{x}}$. In this way,   the corresponding inclusion in Theorem \ref{theo:local:opt} is given with   $ \asub $ and $\tilde{D}^\ast$ replacing the regular subdifferential and coderivative, respectively. 
	
	Therefore, the assumption over the operator $\asub$ and the family $\mathcal{F}_{\bar{x}} $ at the optimal point $\bar{x}$ corresponds to a trade-off between the differentiability of the data $h\in \mathcal{F}_{\bar{x}} $ and the robustness of objects  $\asub$  and $\tilde{D}^\ast$, that is,  while the notion of differentiability is weaker,  larger is the object for which the optimality condition are presented ($\asub$  and $\tilde{D}^\ast$). In that case, the reader could believe that it would be best to directly  assume   that  $h$ satisfies a high level of smoothness  at $\bar{x}$. Nevertheless, for infinite-dimensional applications, such smoothness does not  always hold (see Example \ref{examplenonsmooth} below).
\end{remark}

Similarly  to Theorem \ref{theo:local:opt}, we provide a necessary optimality condition to problem  \eqref{Main_Problem}.  
\begin{corollary}\label{corollary:constraint2}
	In the setting of Corollary \ref{corollary:constraint}, let  $\bar x$ be  a local  optimal solution of  problem \eqref{Main_Problem} and suppose that $h$  is differentiable at $\bar{x}$. Then, there exist $\eta \geq 0$ and	$\lambda^\ast \in  \C^\circ $ such that  $\langle \lambda^\ast , \Phi(\bar {x} ) \rangle   = 1$ and  we have 
	\begin{align}\label{Theo:DC01:local:eq0033}
		0_{X^\ast}  \in      \hat{\partial}\varphi  (\bar{x}) +  \eta \hat{D}^\ast \Phi (\bar x)(\lambda^\ast ) +  N_\Constr(\bar x),
	\end{align}
	provided that the following qualification holds 
	\begin{align}\label{QC01:local02}
		0_{X^\ast} \notin \bigcup\left[     \hat{D}^\ast \Phi (\bar x)( \lambda^\ast )  +N_\Constr(\bar x) 	: \begin{array}{c}
			\lambda^\ast \in  \C^\circ   \text{ such that}\\
			\langle \lambda^\ast , \Phi(\bar {x} ) \rangle   = 1 
		\end{array}	\right].
	\end{align}
	
\end{corollary}

\begin{proof}
	Following the proof of  Corollary \ref{corollary:constraint}, we have that the optimization problem \eqref{Main_Problem01} has a local optimal solution at $\bar x$. Then, by Theorem \ref{theo:local:opt} we get that
	\begin{align}\label{condition0002}
		0_{X^\ast}  \in      \hat{\partial}\varphi  ( \bar{x}) + \operatorname{cone}\left(  \bigcup\left[      \hat{D}^\ast \Phi_\Constr (\bar x)( \lambda^\ast )  	: \begin{array}{c}
			\lambda^\ast \in  \C^\circ   \text{ such that}\\
			\langle \lambda^\ast , \Phi(\bar {x} ) \rangle   = 1
		\end{array}	\right] \right),
	\end{align}
	where $\Phi_{\Constr}$ is defined in \eqref{DEFPHIQ}, and provided that, the following qualification holds:
	\begin{align}\label{condition0001}
		0_{X^\ast} \notin \bigcup\left[     \hat{D}^\ast \Phi_\Constr (\bar x)( \lambda^\ast )  	: \begin{array}{c}
			\lambda^\ast \in  \C^\circ   \text{ such that}\\
			\langle \lambda^\ast , \Phi(\bar {x} ) \rangle   = 1 
		\end{array}	\right].
	\end{align}
	Now, using Lemma \ref{sumerulefrechet:2} we have that \eqref{condition0002} and \eqref{condition0001} reduce  to    \eqref{Theo:DC01:local:eq0033} and \eqref{QC01:local02}, which concludes the proof.
\end{proof}

The final result of this section   shows that the fulfilment of    inclusion   \eqref{THEo:DC01:eq00}, for  all small $\eta\geq 0$, is sufficient for a point     to be a local optimum   of problem \eqref{DC_CONSTRAINT}. 
\begin{theorem}\label{suf_local_conv}
	Let $\bar x$ be   a feasible point of the optimization problem  \eqref{DC_CONSTRAINT} which satisfies the  subdifferential inclusion \eqref{THEo:DC01:eq00}  for all  $\eta $ small enough.   In addition,  suppose 	that   $\C^\circ$ is weak$^\ast$-compact,  that $h$ is continuous at $\bar{x}$ and the following qualification condition holds
	\begin{align}\label{Sub:int00}
		\partial h(\bar{x}) \cap  	\bigcup\left[    {\partial} \left(  \langle \lambda^\ast , \Phi\rangle  + h\right) (\bar{x})  	: \begin{array}{c}
			\lambda^\ast \in  N_{\C}(\Phi(\bar x))  \text{ such that}\\
			\langle \lambda^\ast , \Phi(\bar {x} ) \rangle   = 1
		\end{array}	\right] =\emptyset.	
	\end{align}
	Then, $\bar x$ is a local solution of \eqref{DC_CONSTRAINT}.
\end{theorem}

\begin{proof}
	First, we claim that $\bar x$  satisfies the  subdifferential inclusion \eqref{THEo:DC01:eq00}  with multiplier  $\alpha_1\neq 0$ for all  $\eta\geq 0 $ small enough. 
	Indeed, suppose by contradiction that there are sequences $\eta_n,\eta_n' \to 0^+$  and 
	$$x_n^\ast \in    \partial_{\eta_n}  	h( \bar{x})  \cap   \partial_{\eta'_n} \left( \langle \lambda_n^\ast , \Phi\rangle + h \right) (\bar{x}),$$
	where 	 $	 \eta'_n +1-  \langle \lambda_n^\ast , \Phi(\bar {x} )  \rangle \leq  \eta_n$  and $  \lambda_n^\ast \in \C^\circ $.
	Since $h$ is continuous at $\bar{x}$ we have that $\partial_{\eta_n}  	h( \bar{x}) $ is weak$^\ast$-compact (see, e.g., \cite[Theorem 2.4.9]{MR1921556}). Hence, there exists subnets   (with respect to the weak$^\ast$-topology)   $x_{n_\nu}^\ast$ and $ \lambda^\ast_{n_\nu} $ converging to   $  x^\ast$ and $  \lambda^\ast$, respectively. Then, it is easy to see that $x^\ast \in \partial h(\bar{x}) \cap  \partial \left( \langle \lambda^\ast , \Phi\rangle + h \right) (\bar{x}),$
	which contradicts \eqref{Sub:int00} and proves our claim.
	
	Now, let us denote by $\epsilon_0>0$ a number such that  $\bar x$  satisfies the  subdifferential inclusion \eqref{THEo:DC01:eq00}  with $\alpha_1\neq 0$ for all  $\eta  \in [0,\epsilon_0]$.
	
	Since, $h$ is locally Lipschitz at $\bar x$ there exists a neighbourhood $U$ of $\bar x$ such that  for all $x,y \in U$ and all  $x^\ast \in \partial h(x) $  
	\begin{align*}
		|	\langle x^\ast, y -x \rangle|   + |h(y) -h(x) | \leq \epsilon_0.
	\end{align*}   
	Particularly, for each $y \in U$ \eqref{Main01} holds with $\eta \leq \epsilon_0$. Now, for $y\in U \cap \Phi^{-1}(\C)$, and  repeating the arguments given in the proof of Theorem \ref{THEo:DC01}, we get that $\varphi(\bar{x}) \leq \varphi(y)$, which ends the proof. 
\end{proof}

\begin{remark}
	Let us notice that when $h$ is differentiable at $\bar x$   condition \eqref{Sub:int00} reduces to 
	\begin{align}\label{Sub:int001}
		0_{X^\ast} \notin   	\bigcup\left[   \hat{D}^\ast \Phi (\bar x)( \lambda^\ast ) 	: \begin{array}{c}
			\lambda^\ast \in  \C^\circ   \text{ such that}\\
			\langle \lambda^\ast , \Phi(\bar {x} ) \rangle   = 1 
		\end{array}	\right].	
	\end{align}
	Indeed, if $h$ is differentiable at $\bar x$, we can use the   sum rule \eqref{sumerulefrechet} to get that 
	\[   {\partial} \left(  \langle \lambda^\ast , \Phi\rangle  + h\right) (\bar{x})  	 =\hat{D}^\ast \Phi (\bar x)( \lambda^\ast )  + \nabla h(\bar x).\] 
	Therefore, \eqref{Sub:int00} turns out to be equivalent to \eqref{Sub:int001}.
	
\end{remark}

\section{DC cone-constrained optimization problems}

\label{SECTION_CONE}

This section  addresses to establishing necessary and sufficient conditions for
cone-constrained optimization problems. More precisely, we consider the
following optimization problem

\begin{equation}\label{DCcone_cons} 
	\begin{aligned}
		& \min \varphi (x)      \\
		&\textnormal{{s.t. }}  \Phi(x)\in -\mathcal{K}, 
	\end{aligned}
\end{equation}%
where $\varphi :X \rightarrow \mathbb{R}\cup\{ +\infty\},$ $\Phi:X
\rightarrow Y\cup \{ \infty_{Y}\}$, and $\mathcal{K}\subset 
Y$ is a closed convex cone. 

The approach in this section is slightly different for the one followed
in the previous section where there was a convex abstract constraint
involving a general closed convex set $\mathcal{C}$. More precisely,
we will take advantage of the particular structure of the cone-constraint
$\Phi(x)\in-\mathcal{K}$ in terms of a more suitable supremum
function.   In order to do that we need to introduce the following notion for convex cones.
\begin{definition}\label{defweakcompt}
	Let $\Theta \subseteq Y^\ast$ be a convex closed cone, we say that $\Theta$ is \emph{$w^\ast$-compactly generated} if there exists  a weak$^\ast-$compact and convex set $\mathcal{B} $ such that $$\Theta= \cl^{\ast}\textnormal{cone}( \mathcal{B}  ).$$
	In this case we say that $\Theta$ is \emph{$w^\ast$-compactly generated by} $\mathcal{B}$.
\end{definition}

The next lemma establishes   sufficient conditions to ensure that the polar of a convex cone is $w^\ast$-compactly generated.
\begin{lemma}
	Let $\mathcal{K} \subseteq Y$ be a convex closed cone. Suppose that one of the following conditions is satisfied:
	\begin{enumerate}[label=\emph{\alph*})] 
		\item $Y$ is a normed space.
		\item $\mathcal{K}$ has nonempty interior.
	\end{enumerate}
	Then, $\mathcal{K}^{+}$ is $w^\ast$-compactly generated.
\end{lemma}
\begin{proof}
	\emph{a)} Let us consider the convex compact set $\mathcal{B} :=\{ x^\ast \in  \mathcal{K}^{+}  : \| x^\ast \| \leq 1 \}$. Then, $\mathcal{K}^{+}$ is weakly$^\ast$ compactly generated by $\mathcal{B}$. 
	\emph{b)} Consider an interior point of $\mathcal{K}$, $y_0$, and take  a convex balanced neighbourhood of zero, $V$, such that $y_0 + V \subseteq \mathcal{K}$. Therefore $\mathcal{K}^+ \subseteq \{ x^\ast \in Y^\ast : \langle x^\ast, y_0 \rangle  \geq \sup_{y \in V} \langle x^\ast , y \rangle  \}$. Then, consider $\mathcal{B}:=\{ x^\ast \in \mathcal{K}^+ : \sup_{y \in V} \langle x^\ast , y \rangle  \leq 1  \}$, which is a weak$^\ast$-compact  (convex) set due to the Banach-Alaoglu-Bourbaki Theorem. Moreover, for every $x^\ast \in \mathcal{K}^+$ we have that $\frac{x^\ast }{ |\langle y_0 , x^\ast \rangle  | + 1  } \in \mathcal{B}$, so $\mathcal{K}^+$ is $w^\ast$-compactly generated by $\mathcal{B}$. 
\end{proof}

\subsection{Global optimality conditions}

The next theorem gives necessary and sufficient optimality conditions for
problem \eqref{DCcone_cons}.

\begin{theorem}
	\label{conecons_global} 
	Let $\mathcal{K}$ be a closed convex cone such that $\mathcal{K}^+$ is  w$^\ast$-compactly generated by   $\mathcal{B}$, and assume that $\Phi\in \Gamma_{h}(X,Y, \mathcal{B})$ and $ \varphi \in \Gamma_{h}(X)$ for some function $h \in \Gamma_{0}(X)$. Furthermore,   suppose that one of the following conditions holds:
	\begin{enumerate}[ref=\alph*), label=\alph*)]
		\item\label{conecons_global:itema} the   function $x \mapsto  \varphi(x) + h(x)$ is continuous at some   point of  $\operatorname{dom} \Phi$, 
		\item\label{conecons_global:itemb} the   function $x \mapsto \sup\{  \langle \lambda^\ast , \Phi(x)\rangle  : \lambda^\ast \in   \mathcal{B}\} +h(x)$ is   continuous at some point of $\operatorname{dom} \varphi$.
	\end{enumerate}
	Then, if $\bar{x}$ is a minimum of \eqref{DCcone_cons}, we have that for every $\eta \geq 0$ 
	\begin{equation}\label{coneglobal}
		\partial_{\eta}  	h( \bar{x})  \subseteq \bigcup \Big[    \alpha_1 \partial_{\eta_1 	}( \varphi + h ) ( \bar{x})  +  \alpha_2 \partial_{\eta_2} \left( \langle \lambda^\ast , \Phi\rangle + h \right) (\bar{x})  		\Big],
	\end{equation}
	where the union is taken over all $ \eta_1,\eta_2 \geq 0$, $(\alpha_1,\alpha_2) \in \Delta_2$  and  $  \lambda^\ast \in \mathcal{B} $ such that  
	\begin{align*} 
		\alpha_1   \eta_1+  \alpha_2 ( \eta_2 -  \langle \lambda^\ast , \Phi(\bar {x} )\rangle ) =  \eta.
	\end{align*}
	
	Conversely, assume that $\bar{x}$ is a feasible point of \eqref{DCcone_cons}  and that \eqref{coneglobal} always holds with    $\alpha_1 > 0$, then $\bar{x}$ is a solution of \eqref{DCcone_cons}  relative to $\operatorname{dom} \partial h$.
\end{theorem}
\begin{proof}
	Suppose that $\bar{x}$ is a minimum of \eqref{DCcone_cons}. First, let us notice that  $\Phi(x) \in -\mathcal{K}$ if and only if $\sup_{ y^\ast \in \mathcal{B}}   \langle y^\ast, \Phi\rangle (x) \leq 0$.
	Then, by Lemma \ref{lemmaproblemlema} we have that $\bar x$ is a solution of the DC program 
	\begin{equation*}
		\begin{array}{cc}
			&\min   \max\{ \varphi(x)+h(x)-\alpha,  \sup\limits_{ y^\ast \in \mathcal{B}}   \langle y^\ast, \Phi\rangle (x) +h(x) \} - h(x).\\
		\end{array}
	\end{equation*}
	where $\alpha $ is the optimal value of \eqref{DCcone_cons}. Now, using the notation of  Theorem \ref{THEo:DC01}  we consider  \[\psi_1 (x):  =\varphi(x) + h(x)-\alpha,\;  \psi_2(x):= \sup_{ y^\ast \in \mathcal{B}}   \langle y^\ast, \Phi\rangle (x) +h(x)  \]   and \[\psi (x):=\max\left\{ \psi_1 (x),\psi_2(x)		\right\}. \] Now,  mimicking  the proof of Theorem \ref{THEo:DC01} and    taking into account that, instead of \eqref{THEo:DC01:eq03}, we have 
	\begin{align*} 
		\partial_{\eta_0}  	 \psi_2( \bar{x})  = \bigcup \left[   \partial_{\eta_2} \left( \langle \lambda^\ast , \Phi\rangle + h \right) (\bar{x})  	: \hspace{-0.2cm}\begin{array}{c}  \eta_2 \in  [0,\eta_0], \lambda^\ast \in   \mathcal{B} \text{ and }
			\\ \vspace{-0.2cm}	\\  \sup\limits_{ y^\ast \in \mathcal{B}}   \langle y^\ast, \Phi\rangle (\bar x) \leq \langle \lambda^\ast , \Phi\rangle(\bar x) +   \eta_0-\eta_2 
		\end{array} 	\right],
	\end{align*}
	we have that     \eqref{coneglobal} holds.
	The converse follows as in the  proof of  Theorem \ref{THEo:DC01}, so we omit the details.
\end{proof}

Now, we present a result about optimality conditions of a DC cone-constrained optimization problem with an extra abstract convex constraint. The proof follows similar arguments to the proof of Corollary \ref{corollary:constraint}, but uses Theorem \ref{conecons_global}  instead of Theorem \ref{THEo:DC01}; so we omit the proof.
\begin{corollary}\label{cor:conic:abst:constr}
	Consider the optimization problem
	\begin{equation}\label{Main_Problem:cone:constr}
		\begin{aligned}
			& \min \varphi (x)      \\
			&	\textnormal{s.t. } \Phi(x)\in -\mathcal{K},\;  \\ &\hspace{0.6cm}x \in \Constr,
		\end{aligned}
	\end{equation}%
	where $\Constr $ is closed and convex,    $\mathcal{K}$ is a closed convex cone such that $\mathcal{K}^+$ is  weakly$^\ast$-compact generated by  $\mathcal{B}$, and  $\Phi\in \Gamma_{h}(X,Y, \mathcal{B})$ and $ \varphi \in \Gamma_{h}(X)$ for some function $h \in \Gamma_{0}(X)$. Assume  that there exists a point in  $\Constr$ such that $\varphi + h $ and  $\Phi+ h$ are continuous at this point. 	 Then, if $\bar x$ is an optimal solution of  problem  \eqref{Main_Problem:cone:constr} we have that 
	\begin{align}\label{Cor:cone:eq00:constr}
		\partial_{\eta}  	h( \bar{x})  \subseteq \bigcup \Big[    \alpha_1 \partial_{\eta_1 	}( \varphi + h ) ( \bar{x})   +  \alpha_2 \partial_{\eta_2} \left( \langle \lambda^\ast , \Phi\rangle + h \right) (\bar{x}) +  N_{\Constr}^{\eta_3} (\bar{x})		\Big] \text{ for all } \eta \geq 0,
	\end{align}
	where the union is taken over all $ \eta_1,\eta_2,\eta_3 \geq 0$, $(\alpha_1,\alpha_2) \in \Delta_2$  and $  \lambda^\ast \in \mathcal{B}$ such that  
	\begin{align*}
		\alpha_1   \eta_1+  \alpha_2 ( \eta_2 -  \langle \lambda^\ast , \Phi(\bar {x} )\rangle ) +   \eta_3=  \eta.
	\end{align*}
	
	Conversely, assume that $\bar{x}$ is a feasible point of \eqref{Main_Problem:cone:constr}  and that  \eqref{Cor:cone:eq00:constr} always holds with    $\alpha_1 >0$, then $\bar{x}$ is a solution of \eqref{Main_Problem:cone:constr}  relative to $\operatorname{dom} \partial h$. 
	
\end{corollary}

\subsection{Local optimality conditions}

Now, we focus on necessary and sufficient local optimality conditions
for DC cone-constrained optimization problems. The following two
results provide necessary conditions for optimality of problem
\eqref{DCcone_cons} and for a variant with an additional
abstract convex constraint. The proofs of both results follow similar
arguments to the ones used in Theorem
\ref{theo:local:opt} and Corollary \ref{corollary:constraint2}, respectively;
accordingly, we omit them.

\begin{theorem}\label{conecons_local}
	In the setting of Theorem \ref{conecons_global}, let $\bar x$ be  a local  optimal solution of problem \eqref{DCcone_cons}, and suppose that $h$ is differentiable at $\bar{x}$. Then,   there are  multipliers  $(\alpha_1,\alpha_2)\in\Delta_2$ and  $\lambda^\ast\in  \mathcal{B}$       such that  
	\begin{align*} 
		0_{X^\ast}  \in  \alpha_1  \hat{\partial}\varphi  ( \bar{x})  + \alpha_2 \hat{D}^\ast \Phi (\bar x)( \lambda^\ast ) \text{ and }  \alpha_2 	\langle \lambda^\ast , \Phi(\bar {x} ) \rangle=0.
	\end{align*}
	In addition, if the following qualification  holds
	\begin{align*} 
		0_{X^\ast} \notin \bigcup\left[     \hat{D}^\ast \Phi (\bar x)( \lambda^\ast )  	: \begin{array}{c}\lambda^\ast \in  \mathcal{B}   \text{ such that } 
			\langle \lambda^\ast , \Phi(\bar {x} ) \rangle   = 0	\end{array}	\right],
	\end{align*}
	we have 
	\begin{align*} 
		0_{X^\ast}  \in      \hat{\partial}\varphi  ( \bar{x}) + \operatorname{cone}\left(  \bigcup\left[      \hat{D}^\ast \Phi (\bar x)( \lambda^\ast )  	: \begin{array}{c}\lambda^\ast \in  \mathcal{B}   \text{ such that 	}
			\langle \lambda^\ast , \Phi(\bar {x} ) \rangle   = 0
		\end{array}	\right] \right).
	\end{align*}
	
\end{theorem}

\begin{corollary}\label{Cor:CONe:Local:abst:constr}
	Under the assumptions of Corollary \ref{cor:conic:abst:constr}, let $\bar x$ be  a local  optimal solution of  problem  \eqref{Main_Problem:cone:constr}, and  suppose that $h$ is differentiable  at $\bar{x}$ and $\varphi$ and $\Phi$ are continuous at $\bar{x}$. Then, there exists $\lambda^\ast \in \mathcal{K}^+$  such that
	\begin{align*}
		0_{X^\ast}  \in     \hat{ \partial}   \varphi  ( \bar{x})      +       \hat{D}^\ast \Phi (\bar x)( \lambda^\ast )  	 + N_{\Constr}  (\bar{x})	\text{ and } \langle \lambda^\ast , \Phi(\bar {x} ) \rangle   = 0,
	\end{align*}  
	provided that the following constraint qualification holds	 
	\begin{align*}
		0_{X^\ast}  \notin	  \hat{D}^\ast \Phi (\bar x)( \lambda^\ast )  	 + N_{\Constr}  (\bar{x}), \text{ for all } \lambda^\ast \in \mathcal{B}.	\end{align*}
\end{corollary}

Similarly to Theorem \ref{suf_local_conv}, we provide sufficient conditions for  local optimality in terms of \eqref{coneglobal}.

\begin{theorem}\label{suf_local_cone} 
	Let $\bar x$ be   a feasible point  problem  \eqref{DCcone_cons}  satisfying the  subdifferential inclusion \eqref{coneglobal}  for all  $\eta $ small enough.  Additionally, suppose that $h$ is continuous at $\bar{x}$ and the following qualification holds
	\begin{align}\label{Sub:int002}
		\partial h(\bar{x}) \cap  	\bigcup\left[    {\partial} \left(  \langle \lambda^\ast , \Phi\rangle  + h\right) (\bar{x})  	: \begin{array}{c}
			\lambda^\ast \in  \mathcal{B}  \text{ such that } 
			\langle \lambda^\ast , \Phi(\bar {x} ) \rangle   = 0
		\end{array}	\right] =\emptyset.		
	\end{align}
	Then, $\bar x$ is a local solution of \eqref{DC_CONSTRAINT}.
\end{theorem}

\begin{remark}
	Notice that  when $h$ is differentiable   at $\bar{x}$
	condition \eqref{Sub:int002} leads us to 
	\begin{align*}
		0_{X^\ast} \notin   	\bigcup\left[    \hat{ D}^\ast \Phi (\bar x)( \lambda^\ast )  	: \begin{array}{c}
			\lambda^\ast \in  \mathcal{B}    \text{ such that }
			\langle \lambda^\ast , \Phi(\bar {x} ) \rangle   = 0 \hspace{-0.15cm}
		\end{array}\right].
	\end{align*}
	
\end{remark}

\section{Applications to   mathematical programs problems}\label{Sect_Applications}
In this section we provide some applications of the theory developed in the previous  sections.
\subsection{Infinite programming}\label{sub_section_semiinfinite}
We consider the optimization problem
\begin{equation}\label{probleseminf}
	\begin{array}{l}
		\min \varphi(x)   \\
		\textnormal{{s.t. }} \phi_t(x)\leq 0, \;t\in T,
	\end{array}
\end{equation}%
where $T$ is a    locally compact   Hausdorff space,  the functions $\phi_t :   X\to \R$, $t\in T$,  are      such that, for all $x\in X$, the function, $t\mapsto  \phi_t(x)\equiv\phi(t,x)$ is continuous with compact support,  and $\varphi: X\to \overline{\mathbb{R}}$. Problem  \eqref{probleseminf} corresponds to the class of infinite programming problems (called semi-infinite when $X$ is finite dimensional); we refer to \cite{MR2295358} for more details about the theory. 

The space of continuous functions defined on $T$  and  with compact support is denoted  by $\mathcal{C}_c(T)$. A finite measure $\mu : \Borel{T} \to [0,+\infty)$, where $\Borel{T}$ is the Borel $\sigma$-algebra,   is called \emph{regular} if for 
every $A \in \Borel{T} $
\begin{equation*}
	\begin{aligned}
		\mu(A)=\inf\left\{ \mu(V)   : \hspace{-0.05cm}V \text{ is open and  }V \supseteq A   \right\}=\sup \left\{ \mu(K)   : \hspace{-0.05cm}K \text{ is compact and  }K \subseteq A   \right\}.
	\end{aligned}
\end{equation*}
We denote by  $\Radon{T}$ the set of all (finite) regular Borel measures.  Let us recall (e.g., \cite[Theorem 14.14]{MR2378491}) that   the dual of $\mathcal{C}_c(T)$, endowed with the uniform norm,   can be identified as the linear space generated by   $\Radon{T}$. 

The following result provides necessary optimality conditions for problem \eqref{probleseminf} using a cone representation in the space  $\mathcal{C}_c(T)$.
\begin{theorem}\label{Theo:6.1}
	Let  $X$   be  a  Banach space and $T$ be a locally compact Hausdorff space. Suppose that $\varphi, \phi_t\in \Gamma_{h}(X)$,  $t\in T$, for some  function $h\in \Gamma_0(X)$, and assume that the function  $x\mapsto \inf_{t\in T} \phi_t(x)$ is locally bounded from below.   Let   $\bar x$ be   a local optimum  of problem \eqref{probleseminf},   and assume  that $h$ and $\phi_t  $, $t\in T$,  are differentiable  at $\bar{x}$,  and that there are  $\ell,   \epsilon > 0$ such that 
	\begin{align}\label{Lipschitz}
		|\phi_t (y) -\phi_t(\bar{x}) | &\leq \ell  \| y- \bar{x}\|, \quad \text{ for all } (t,y)\in T\times \mathbb{B}_\epsilon(\bar{x}).
	\end{align}
	Then, there exists  $\mu \in \Radon{T}$ such that $\supp \mu \subseteq T(\bar x):=\{ t\in T : \phi_t(\bar{x} )=0\}$,  
	\begin{align*}
		- \int_T \nabla  \phi (t, \bar{x}) \mu(dt)\in \hat \partial  \varphi(\bar{x} ) \text{ and }\int_T \phi_t(\bar x)\mu(dt) =0,
	\end{align*}
	provided that  \begin{align*}
		0_{X^\ast} \notin   \bigcup\left\{   \int_T \nabla  \phi (t, \bar{x}) \nu(dt)  	: \begin{array}{c}\nu \in \Radon{T} \text{ such that } \nu(T)=1,  \\  \vspace{-0.3cm} \\\supp \nu \subseteq T(\bar x)    \text{ and }
			\int_T \phi_t(\bar x) \nu(dt)   = 0	\end{array}	\right\},
	\end{align*} 
	where the integrals are in the sense of  Gelfand (also called $w^{\ast}$-integrals).  
\end{theorem}
\begin{remark}(before the proof)
	It is important to recall that a mapping $x^\ast : T \to X^\ast$ is \emph{Gelfand integrable} if for every $x \in X$, the function $t\mapsto \langle x^\ast(t),x\rangle $ is integrable. In that case, the integral $  \int_T   x^\ast(t) \nu(dt)  $ is well-defined as  the unique element of $X^\ast$ such that 
	\begin{align*}
		X\ni 	x \mapsto \int_T    \langle x^\ast(t),  x\rangle  \nu(dt).
	\end{align*}
	We refer to \cite[Chapter II.3,	p. 53]{MR0453964}  for more details. 
\end{remark}
\begin{proof}
	First, let us define $\Phi: X\to \mathcal{C}_c(T)$  as the evaluation function $\Phi(x): T \to \R$ given by 
	$\Phi(x)(t):= \phi_t(x)$, which is well-defined thanks to our assumptions. Given  $\mathcal{K}:=\{ x \in \mathcal{C}_c(T) : x(t) \geq 0 \text{ for all }t\in T   \}$, by \cite[Theorem 14.12]{MR2378491} we have  $\mathcal{K}^+= \Radon{T}$, and it is easy to see that $\mathcal{K}^+$ is $w^\ast$-compactly generated by $\mathcal{B}:=\{ \mu \in \Radon{T} : \mu(T)=1   \}$  (in the sense of Definition \ref{defweakcompt}). Let us first prove two claims.
	\newline 
	\noindent	\emph{ {Claim 1:}  The   function $x \mapsto \sup\{  \langle \nu , \Phi(x)\rangle  :\nu \in   \mathcal{B}\} +h(x)$ is   continuous, and the mapping     $\Phi$ belongs  to $\Gamma_{h}(X,\mathcal{C}_c(T),\mathcal{B})$.}\\
	To this purpose,  fix a measure $\nu \in  \mathcal{B}$. By the assumptions, the  function $x \mapsto\phi_{t}(x)+h(x)$ is convex for all $t\in T$, so integration over $T$ with respect to $\nu$  preserves the convexity on $X$ (see, e.g., \cite{Mordukhovich2021,MR3947674,MR4062793,MR4261271,MR4350897,MR236689,MR310612}); hence,  the function $\langle \nu , \Phi\rangle + h$ is convex. Moreover, consider a sequence  $x_k \to x$. Since, the function  $ x\mapsto \inf_{t\in T} \phi_t(x)$ is locally bounded from below and $h \in \Gamma_{0}(X)$,  we can take $\alpha \in \R$ and $k_0\in \N$ such that 
	$\phi_t(x_k ) + h(x_k)\geq \alpha$, for all $t\in T$ and all $k \geq k_0$. Then,   Fatou's lemma and the lower semicontinuity of the functions $\phi_t  + h$, $t\in T$,  yield
	\begin{align*}
		\langle \nu , \Phi(x)\rangle + h(x) &= \int_{T} \left( \phi_t(x) + h(x) \right) \nu(dt) \leq \int_{T}  \liminf_{k\to \infty}\left( \phi_t(x_k) + h(x_k) \right)\nu(dt) \\
		& \leq \liminf_{k\to \infty} \int_{T}  \left( \phi_t(x_k) + h(x_k) \right)\nu(dt)  = \liminf_{k\to \infty} \left( \langle \nu , \Phi(x_k)\rangle + h(x_k) \right),
	\end{align*}
	showing the lower semicontinuty of  $x \mapsto \langle \nu , \Phi \rangle(x)+h(x)$, and that, consequently the function $\Phi$  belongs to $ \Gamma_{h}(X,\mathcal{C}_c(T),\mathcal{B})$. Finally, since the  function   $x \mapsto \sup\{  \langle \nu , \Phi(x)\rangle  :\nu \in   \mathcal{B}\} +h(x)$ is convex, lsc and finite valued because $\mathcal{B}$ is $w^\ast$-compact, it is also continuous (recall that $X$ is a Banach space).
	\newline 
	\noindent	\emph{ {Claim 2:}  For every $h \in X$ the function    $t\mapsto \langle \nabla  \phi_t(\bar x), h \rangle $ is measurable  and, for every $\nu \in \mathcal{B}$, we have that  \[\hat{D}^\ast \Phi (\bar x) (\nu) \subseteq  \left\{  \int_T \nabla  \phi_t (\bar{x}) \nu(dt)\right\}. \] }\\
	
	Fix $h\in X$. Since the functions $\phi_t$, $t\in T$, are differentiable at $\bar{x}$, we get that   $\langle \nabla  \phi_t(\bar x), h \rangle = \lim_{ k \to \infty }  k\left( { \phi_t( \bar x + k^{-1} h)  - \phi_t(\bar x)}\right)$. Particularly, the function $t\mapsto \langle \nabla  \phi_t(\bar x), h \rangle$ is measurable as it is the pointwise  limit  of a sequence of measurable functions. Moreover, by \eqref{Lipschitz} we get that   $\langle \nabla  \phi_t(\bar x), h \rangle \leq \ell \|h\|$, for all $t\in T$, which shows the integrability  and, consequently, the Gelfand integral is well-defined (see, e.g., \cite{MR2378491,MR0453964}). Finally, let $x^\ast \in  D^\ast \Phi (\bar x) (\nu)$, the definition of regular  subdifferential, with $S=\{h\} \in \beta$, implies that   
	\begin{align}\label{inequalitylinearform}
		\langle x^\ast ,h \rangle \leq \lim\limits_{k \to +\infty  } \frac{ \langle \nu ,  \Phi (\bar x + k^{-1} h )   - \Phi (\bar x)  \rangle }{k^{-1}} = \int_T \langle \nabla  \phi (t, \bar{x}),h \rangle  \nu(dt), \quad \forall h \in X,
	\end{align}
	where in the last equality we use Lebesgue's dominated convergence theorem, which can be applied thanks to  \eqref{Lipschitz}.The proof of this claim ends by considering $h$ and $-h$ in \eqref{inequalitylinearform}.
	
	Finally, observe that, by \cite[Lemma 12.16]{MR2378491}, any measure $\nu \in \Borel{T}$ such that $\int_T \phi_t(\bar x) \nu(dt) = 0$ satisfies that $\supp \nu \subseteq T(\bar x)$. Then, applying Theorem \ref{conecons_local} we get the desired result. 
\end{proof}

\subsection{Stochastic programming}
Before introducing our   optimization problem in this subsection let us give some additional  notations. In the sequel, $X$ is  a separable  Banach space, $Y$  a general locally convex space, and  $(\Omega,\mathcal{A}, \mu)$   a  complete $\sigma$-finite measure space. A set-valued mapping    $M: \Omega \tto X$ is said to be \emph{measurable}  if for every open set $U \subseteq X$, we have  $ \{ \omega  \in  \Omega : M(\omega ) \cap U \neq \emptyset \} \in \mathcal{A}$.  
A function $\varphi: \Omega \times X \to \R\cup\{ +\infty\}$ is said to be a \emph{normal integrand} provided the set-valued mapping  $\omega \mapsto \epi \varphi_\omega:=\{ (x,\alpha) \in X\times \R : \varphi_\omega(x):=\varphi(\omega, x)\leq \alpha \} $ is measurable with closed values. In addition, $\varphi$   is a \emph{convex normal integrand} if $\varphi_\omega $ is convex for all $\omega \in \Omega$.

Given a set-valued mapping $S: \Omega \tto X^\ast$ we define the \emph{(Gelfand) integral} of $S$ by 
\begin{align*}
	\int_{\Omega} S(\omega) \mu (d\omega) := \left\{   \int_\Omega x^\ast(\omega) \mu(d\omega)   : \begin{array}{c}
		x^\ast \text{ is Gelfand integrable and } \\
		x^\ast(\omega) \in S(\omega) \text{ a.e. } \omega \in \Omega
	\end{array}\right\}.
\end{align*}
We refer to \cite{MR0453964,MR0117523,MR0467310,MR1491362} for more details about the theory of measurable multifunctions and integration on Banach spaces. 

Given a normal integrand $\varphi  : \Omega \times X \to  \R\cup\{ +\infty\}$ we define the \emph{integral functional}  (also called \emph{expected functional}) associated to $\varphi$ by $ \Intf{\varphi}: X\to \R \cup \{ +\infty,-\infty\}$ defined as
\begin{align}\label{IntegralFunctional}
	\Intf{\varphi}(x):=\int\limits_{\Omega} \varphi_\omega(x) \mu(d\omega) := \int\limits_{\Omega} \max\{ \varphi_\omega (x) , 0\} \mu(d\omega) + \int\limits_{\Omega}\min\{  \varphi_\omega(x), 0\} \mu(d\omega),
\end{align}
with the inf-addition convention $+\infty +(-\infty) = +\infty$.

Finally, a normal  integrand $\varphi$ is \emph{integrably differentiable} at $\bar{x}$ provided that $\Intf{\varphi}$ is differentiable at $\bar{x}$ and the following integral formula holds
\begin{align}\label{integralformula}
	\nabla   \Intf{\varphi}( \bar{x} ) = \int\limits_{\Omega} \nabla  \varphi_\omega (\bar{x}) \mu (d\omega).
\end{align}

The next example shows that the last notion makes sense for integral mappings since its   smoothness cannot be taken for granted even when all  data are  smooth.  
\begin{example}\label{examplenonsmooth}
	It is important to mention here that the integral functional $	\Intf{\varphi} $, for a normal integrand $\varphi$, could   fail to be Fr\'echet differentiable even when the data functions $\varphi_\omega$, $\omega \in \Omega$, are Fr\'echet differentiable. Let us consider the measure space  $(\N ,\mathcal{P}(\N), \mu)$, where the $\sigma$-finite measure is given by the counting measure $\mu(A):=|A|$, and  the Banach  space $X=\ell_1$. Next, consider   the convex normal integrand function $\varphi (n,  x) :=   |  \langle x, e_n\rangle  |^{1 +\frac{1}{n}}   $, where $\{ e_1, e_2, \ldots, e_n,\ldots\}	$ is the canonical basis of $\ell_1$. It  has been shown in \cite[Example 2]{MR3947674} that   $\Intf{\varphi}$ is Gateaux differentiable at any point and the integral formula \eqref{integralformula} holds for the Gateaux derivative.	Nevertheless,   as it was also proved in that paper, the function $\Intf{\varphi}$ fails to be  Fr\'echet differentiable at zero.
\end{example}

Now, we extend  a classical formula for the subdifferential  of convex normal integrand functions to the case of nonconvex normal integrands.   This result is interesting in itself, and  for that reason, we present it as  an independent proposition. 

\begin{proposition}\label{pisub}
	Let $\bar x\in X$, and $\varphi : \Omega \times X \to  \R\cup\{ +\infty\}$ be a normal integrand. Suppose that  $\varphi_\omega \in \Gamma_{h_\omega}(X)$ for some  convex normal integrand $h$ such that $\operatorname{dom}  \Intf{\varphi} \subseteq  \operatorname{dom} \Intf{h}$. Then, $\Intf{\varphi} \in \Gamma_{\Intf{h} }(X)$ provided that $\Intf{\varphi}$ is proper. In addition, suppose that $h$  is integrably  differentiable  at $\bar{x} $ and  the functions  $ \Intf{\varphi}$ and $\varphi_\omega$, $\omega \in\Omega$, are continuous at some common point.  	Then,
	
	\begin{align}\label{Formula_pisub}
		\hat{ \partial} \Intf{\varphi}(\bar x) = \int_{\Omega} \hat{\partial} \varphi_\omega (\bar x) d\mu + N_{ \operatorname{dom} \Intf{\varphi} }(\bar x).
	\end{align}
	
\end{proposition}
\begin{proof}
	Let us consider the convex normal integrand $\psi := \varphi +h$. By our assumptions $\operatorname{dom} \Intf{\psi} = \operatorname{dom} \Intf{\varphi}$ and $\Intf{\psi}= \Intf{\varphi} + \Intf{h}$. Consequently,  $\Intf{\psi}$ is proper, entailing   that $\Intf{\varphi} \in \Gamma_{\Intf{h} }(X)$.  Then, by \cite[Theorem 2]{MR3947674} we have that $ \partial \Intf{\psi}(\bar x)=  \int_{\Omega} {\partial} \psi_\omega (\bar x) d\mu + N_{ \operatorname{dom} \Intf{\psi} }(\bar x).$ 
	Now, by Lemmas \ref{sublemma01} and \ref{lemma:sumrule}, and the integral formula  \eqref{integralformula} we have that 
	\begin{align*}
		\partial \Intf{\psi}(\bar x) & =  	\hat{ \partial} \left( \Intf{\varphi} + \Intf{h} \right) (\bar{x}) = 	\hat{ \partial}  \Intf{\varphi}(\bar x) + \nabla   \Intf{h}( \bar{x} ), \\
		\int_{\Omega} {\partial} \psi_\omega (\bar x) d\mu &  =\int_{\Omega} \hat{\partial} \varphi_\omega (\bar x) d\mu   + \nabla   \Intf{h}( \bar{x} ),
	\end{align*}
	which implies that \eqref{Formula_pisub} holds.
\end{proof}
\begin{remark}[On the use of Gelfand integrals]
	The above result does not require that  $\varphi$   be locally Lipschitzian at $\bar{x}$ as in other classical results about differentiation of nonconvex integral functionals  (see, e.g., \cite{MR4062793,Mordukhovich2021} and the references therein). Consequently, we cannot expect that the formula \eqref{Formula_pisub} holds for the \emph{Bochner integral} (see, e.g., \cite[Definition 11.42]{MR2378491}). Indeed, adapting \cite[Example 1]{MR3947674}, let us consider the measure space  $(\N ,\mathcal{P}(\N), \mu)$, where   $\mu(A):=\sum_{j\in A}2^{-j}$, the Hilbert space $X=\ell_2$,	 and the normal integrand function $\varphi (n,  x) := 2^{n} \left(  \langle x, e_n\rangle  \right)^2  - \| x\|^2 $, where $\{ e_1, e_2, \ldots, e_n, \ldots \}$ is the canonical basis of $\ell_2$. Clearly, the integrand $\varphi$ satisfies all the assumptions of Proposition \ref{pisub} at any point, therefore \eqref{Formula_pisub} holds. 
	Nevertheless, the function $n \mapsto \nabla \phi_n (  x) = 2^{n+1}\left(  \langle x, e_n\rangle  \right) e_n - x $ is not always integrable  in the Bochner sense  because, otherwise, the function $n \mapsto 2^{n+1}\left(  \langle x, e_n\rangle  \right) e_n$ must be integrable (see, e.g., \cite[Theorem 11.44]{MR2378491}). Indeed, we always have
	\[ \int_{\N} \| 2^{n+1}\left(  \langle x, e_n\rangle  \right) e_n \|  d\mu = 2 \sum_{n\in \N} |x_n|.\]
	Nonetheless, the right-hand side is equal to $+\infty$ for   $x=(\frac{1}{n})_{n \in\N} \in \ell_2\backslash \ell_1$.
\end{remark}

%
Given the normal integrand $\varphi :\Omega\times  X\to \R\cup\{ +\infty\}$, the  mapping $\Phi: X \to  Y$ and the nonempty convex closed set   $\C\subseteq Y$, we consider the    problem
\begin{equation}\label{ProblemStoch}
\begin{array}{l}
	\min\Intf{\varphi}(x) \;\\
	\textnormal{s.t. }\;   \Phi( x) \in \C.  
\end{array}
\end{equation}%
\vspace{-0.4cm}
\begin{theorem}
Given problem \eqref{ProblemStoch},	let us assume that   $\C^\circ$ is weak$^\ast$-compact,	let $\varphi_\omega \in \Gamma_{h_\omega}(X)$, $\omega \in \Omega$,   for some  convex normal integrand $h$  such that $\operatorname{dom}  {h_\omega} =\operatorname{dom} \Intf{h}=X$, $\omega \in \Omega$, and let $ \Phi\in \Gamma_{g}(X,Y,\C^\circ)$ for some  function $g\in \Gamma_0(X)$ such that $\operatorname{dom} \Phi=X$.  Let    $\bar x\in \inte\left( \operatorname{dom} \Intf{\varphi}\right)$ be   a local optimal solution of problem \eqref{ProblemStoch}   and assume  that $h$ is integrably differentiable at $\bar{x}$  and $g$ is   differentiable   at $\bar{x}$. Then, 	we have 
\begin{align}\label{THEO:STOC:01:EQ01}
	0_{X^\ast}  \in      \int_{\Omega} \hat{\partial} \phi_\omega (\bar x) d\mu  + \operatorname{cone}\left(  \bigcup\left[      \hat{D}^\ast \Phi (\bar x)( \lambda^\ast )  	: \begin{array}{c}
		\lambda^\ast \in  \C^\circ   \text{ such that}\\
		\langle \lambda^\ast , \Phi(\bar {x} ) \rangle   = 1 
	\end{array}	\right] \right),
\end{align}
provided that  the following qualification condition holds
\begin{align*} 
	0_{X^\ast} \notin \bigcup\left[     \hat{D}^\ast \Phi (\bar x)( \lambda^\ast )  	: \begin{array}{c}
		\lambda^\ast \in  \C^\circ   \text{ such that}\\
		\langle \lambda^\ast , \Phi(\bar {x} ) \rangle   = 1 
	\end{array}	\right].
\end{align*}
\end{theorem}

\begin{proof}
Let us define $h_2:= \Intf{h} +g$, and  notice that $\operatorname{dom} h_2= X$ and $\Intf{\varphi}\in  \Gamma_{h_2} (X)$ (see Proposition \ref{pisub}), and $\Phi \in \Gamma_{h_2}(X,Y,\C^\circ)$. Moreover, since $X$ is a Banach space and $\operatorname{dom} \Phi =X$, we have that  condition \ref{THEo:DC01:itemb} of Theorem \ref{THEo:DC01} holds. Moreover,  $h_2$ is   continuous ($\operatorname{dom} h_2 =X$) and   differentiable  at $\bar{x}$. Then, by Theorem \ref{theo:local:opt} and Propostion  \ref{pisub} we have that \eqref{THEO:STOC:01:EQ01} is satisfied.
\end{proof}

\subsection{Semidefinite Programming}
In this subsection, we consider the optimization problem

\begin{equation}\label{problesemidef}
\begin{array}{l}
	\min \varphi (x)      \\
	\textnormal{s.t. } \Phi(x) \preceq  0,\;  \\
	\hspace{0.6cm}x \in \Constr.
\end{array}
\end{equation}
where $\varphi: X\to \overline{\mathbb{R}}$, $\Phi: X\to \Sp\cup \{ \infty_{\Sp} \}$, with  $\Sp$ being  the set of $p\times p$ symmetric (real) matrices, and $\Constr \subseteq X$ a nonempty convex and closed set of $X$. Here, $A  \preceq 0 $   ($A \succeq 0$, respectively) means that the matrix $A $ is negative semidefinite (positive semidefinite, respectively). We recall that $\Sp$ is a Hilbert space with the inner product $\langle A, B \rangle :=\Tr(AB)$, where $\Tr $ represents the trace  operator (see, e.g., \cite{MR1756264}).  We recall  that, for any symmetric matrix  $A$, $\Tr(A)= \sum_{i=1}^p \lambda_i(A)$,  where $\lambda_1(A)\geq  \ldots \geq \lambda_p(A)$ are the eigenvalues of $A$ (see, e.g., \cite[Theorem 2.6.6]{MR1908225}).

Classical studies of problem \eqref{problesemidef}   suggested  imposing some degree of convexity to  the function $\Phi$, more precisely,  the so-called \emph{matrix convexity} (see, e.g., \cite[Section 5.3.2]{MR1756264}).  This notion is equivalent to assuming that, for every $v\in \mathbb{R}^p$ with $\|v\| =1$, the function $x \mapsto v^\top \Phi(x) v$ is convex, where $v^\top$ is the transpose vector of $v$ (see, e.g. \cite[Proposition 5.72 ]{MR1756264}). In the spirit  of DC optimization, a natural extension of such concept is given by the following notion. We say  that    $\Phi: X\to \Sp \cup\{ \infty_{\Sp}\}$  is a \emph{DC  matrix-mapping}, with control function $h \in \Gamma_0(X)$,  if $\operatorname{dom} \Phi \subseteq \operatorname{dom} h$ and, for  all  $v\in \mathbb{R}^p$ with $ \sum_{i=1}^p v_i^2 =1$, the mapping $x \mapsto v^\top \Phi(x) v + h(x)$ belongs to $\Gamma_{0}(X)$, where
\begin{align*}
v^\top \Phi(x) v :=\left\{  
\begin{array}{cc}
	v^\top \Phi(x) v, &  \text{ if } x\in \operatorname{dom} \Phi,\\
	+\infty, &  \text{ if } x\notin \operatorname{dom} \Phi.
\end{array}
\right.
\end{align*}
Moreover, it is well-known that problem \eqref{problesemidef} can be reformulated similarly to problem \eqref{Main_Problem:cone:constr} involving the cone $\Sp_+$ of positive semidefinite matrices. Due to this observation,  and the topological  structure  of $\Sp$,  another  natural assumption in our framework is   the $\Phi \in \Gamma_{h}(X,\Sp, \mathcal{B})$, for some convex control function $h$ and   $\mathcal{B}$ being a  $w^\ast$-compact  generator of $\Sp_+$ (see  Definition \ref{defweakcompt}).  The following result formally establishes that both notions,  {DC  matrix-mapping} and  {$ \mathcal{B}$-DC mappings}, coincide.

\begin{proposition}\label{Proposit}
Consider an  lcs space $X$, $\Phi: X\to \Sp \cup\{ \infty_{\Sp}\}$,  and $h\in \Gamma_{0}(X)$ with $\operatorname{dom} h \subseteq \operatorname{dom} \Phi$. Then, the following are equivalent:
\begin{enumerate}[label=\alph*)]
	\item $\Phi \in \Gamma_h(X,\Sp, \mathcal{B})$, where   $\mathcal{B}:=\{  A \in \Sp_+ : \Tr(A) =  1  \}$ 
	\item  $\Phi$ is a { DC matrix-mapping}  with control $h$.
\end{enumerate}
Moreover, in such  case, we have that for all $x\in X$, the following equality holds
\begin{align}\label{eqsup}
	\max\{ \langle A, \Phi \rangle (x) + h(x) : A\in \mathcal{B}   \} =  \max  \{ v^\top \Phi(x)  v + h(x) :  \sum_{i=1}^p v_i^2 =1 \}
\end{align}
\end{proposition}
\begin{proof}
Let us suppose that $a)$ holds, and fix ${u}\in \mathbb{R}^p$ with $\| u\|=1$, where $\| \cdot \|$ is the Euclidean norm. Then, consider the symmetric  matrix $A=u u^\top =(u_i u_j)_{ij}$. We have that $v^\top A v =  ( \langle u,v \rangle  )^2$, for all $v\in \mathbb{R}^p$, which shows that $A$  is positive semidefinite and  $\Tr(A)= \| u\|^2=1$; hence,  $A \in \mathcal{B}$. Finally, $  \langle A, \Phi \rangle  (x)=\langle A, \Phi(x) \rangle  =  u^\top \Phi(x) u$, for every $x\in \operatorname{dom} \Phi$, which shows that the function $x\mapsto u^\top \Phi(x)u +h(x)$ is  convex proper and lsc because $\Phi \in \Gamma_h(X,\Sp, \mathcal{B})$.

Now, suppose that $b)$ holds, and consider $A \in \Sp_+ $ with  $\Tr(A)=1$. Employing its  spectral decomposition,  we write $A= PDP^\top = \sum_{i=1}^p \lambda_i(A) v_i v_i^\top$, where $P$ is an orthogonal matrix whose columns are  $v_i \in\mathbb{R}^p$, $i=1,\ldots, p$,   and   $D$ is the  diagonal  matrix formed by  $ \lambda_1(A),\ldots,  \lambda_p(A)$. Then, 
\begin{align*}
	\langle A, \Phi(x) \rangle= \sum\limits_{i=1}^p \lambda_i(A) \langle v_i v_i^\top , \Phi(x) \rangle  = \sum\limits_{i=1}^p \lambda_i(A)  v_i^\top  \Phi(x)v_i, \text{  for all } x\in \operatorname{dom} \Phi.
\end{align*}
Next, using the fact that $\Tr(A)=1$ and $\lambda_i(A) \geq 0$, we get that  	\begin{equation}\label{eq001sup}
	\langle A, \Phi \rangle(x) + h(x)  = \sum_{ i=1 }^p \lambda_i (A)  \left(  v_i^\top  \Phi(x)v_i+ h(x)    \right) ,
\end{equation} 
which shows the desired convexity   as well as the lower semicontinuity  of the function $x \mapsto 	\langle A, \Phi\rangle(x)  + h(x)$.

Finally, \eqref{eqsup} remains to be proved. On the one hand,  from the fact that $A=u u^\top$ is  positive semidefinite,  we get that 
$$  v^\top \Phi(x)  v + h(x) \leq  \max\{ \langle A, \Phi \rangle (x) + h(x) : A\in \mathcal{B}   \}, $$
which proves the inequality $\geq $ in  \eqref{eqsup}. On the other hand, for a given matrix $A\in \mathcal{B}$, and taking into account \eqref{eq001sup}, we have that 
\begin{align*}
	\langle A, \Phi(x) \rangle + h(x) & \leq \sum_{ i=1 }^p \lambda_i (A) \max  \{ v^\top \Phi(x)  v + h(x) :  \sum_{i=1}^p v_i^2 =1 \}\\
	&= \max  \{ v^\top \Phi(x)  v + h(x) :  \sum_{i=1}^p v_i^2 =1 \}.
\end{align*}
and we are done.
\end{proof}

The following proposition establishes some sufficient conditions ensuring that $\Phi$ is  a  DC matrix-mapping.
\begin{proposition}
Let     $\Phi: X\to \Sp $  be a mapping with $\Phi(x)=(\phi_{ij}(x))$ for   $\phi_{ij} \in \Gamma_{h_{ij}}(X,\R,[-1,1])$  for some control function $h_{ij}$, $i,j=1,\ldots,p$. Then, $\Phi$ is a {DC matrix-mapping} with control   $h:=\sum_{ij} h_{ij}$. 

\end{proposition}
\begin{proof}
Let us notice that for every $v\in \mathbb{R} $ with $\| v\|=1$, we have that
\begin{align*}
	v^\top \Phi(x) v + h(x) = \sum_{i,k=1}^p \left(  v_iv_k \phi_{ij}(x) + h_{ij} (x)   \right), \text{ for all  }x\in X.
\end{align*}
Therefore, the function $x\mapsto v^\top \Phi(x) v  + h(x)$ is convex and lower semicontinuous, which yields that  $\Phi$ is a {DC matrix-mapping} with control   $h$.
\end{proof}

Although the notion of {DC matrix-mapping} has a simpler description in terms of quadratic forms $  v^\top \Phi(x) v $, the set $\Gamma_h(X,Y,\mathcal{B})$   in Definition \ref{definition_DCfunction} provides enhanced properties depending on  the choice of   scalarizations $\mathcal{B}$ and, consequently, better properties of operations with the matrix $\Phi$. 

Given a mapping $\Phi : X\to \Sp \cup\{ \infty_{\Sp}\}$, we define its \emph{$k$-th  eigenvalue function} by $\lambda_k^\Phi : X\to \overline{\mathbb{R}}$ given by 
\begin{align*}
\lambda_k^\Phi (x)&:=\left\{ \begin{array}{cl}
	\lambda_k(\Phi(x)),  &\text{ if  } x\in \operatorname{dom} \Phi,\\
	+\infty ,  &\text{ if  } x\notin \operatorname{dom} \Phi.
\end{array}  \right. 
\end{align*}  
Moreover, we define the sum of the first $k$ eigenvalue  functions by  $ \Lambda_k^\Phi : X\to \overline{\mathbb{R}}$  given by $\Lambda_k^\Phi(x):=\sum_{j=1}^{k} \lambda_j^\Phi (x)$.

The following proposition gives  sufficient conditions to ensure that the above functions are DC.

\begin{proposition}
Let     $\Phi: X\to \Sp \cup\{ \infty_{\Sp}\}$ and $h \in \Gamma_0(X)$.
\begin{enumerate}[label=\alph*)]
	\item If $\Phi$ is a  DC matrix-mapping with control $h$, then its  largest eigenvalue  function, $\lambda_1^\Phi$, belongs to $\Gamma_h(X)$.
	\item If  $\Phi \in \Gamma_h(X,\Sp,\mathcal{P}_k)$, where $\mathcal{P}_k :=\{ P \in \Sp : P \text{ has rank  } k \text{ and } P^2=P  \}$, then  $\Lambda_k^\Phi \in  \Gamma_{h}(X)$.  
	\item If $X$ is a Banach space, $\Phi \in \Gamma_{h}(X,\Sp,\mathbb{B}_{\Sp})$  with $\operatorname{dom} \Phi =X$, and  $\mathbb{B}_{\Sp}$ is the unit ball on $\Sp$,   then all the    eigenvalue  functions $\lambda_k^\Phi$  are a difference of convex functions. 
\end{enumerate}

\end{proposition}
\begin{proof}
Let us recall that  $  \lambda_1^\Phi (x)  = \sup\{ v^\top \Phi(x) v : \|v\|=1 \} $, which shows statement \emph{a)}. Second,  let us notice that  \[\Lambda_k^\Phi(x) +h(x)= \max\{  \langle P, \Phi (x) \rangle +h(x): P \in \mathcal{P}_k \}\] (see, e.g.,  \cite[Exercise 2.54]{MR1491362}), so $\Lambda_k^\Phi \in  \Gamma_{h}(X)$, which shows \emph{b)}.    Finally, to prove  \emph{c)}, we have that $\lambda_k^\Phi  (x)= \Lambda_k^\Phi(x) - \Lambda_{k-1}^\Phi(x)$ for all $k \geq 2$, so it is a difference of convex functions. 
\end{proof}

Finally, let us go back to  problem \eqref{problesemidef}. 

\begin{theorem}\label{Opt:Cond:semidefi}
Let $\varphi\in \Gamma_{h}(X)$ and $\Phi$ be a  DC matrix-mapping with control $h$  such that $x \mapsto \varphi(x) + h(x)$ is continuous at some point of $\operatorname{dom} \Phi$. Let $\bar x$ be a  local optimal solution of the optimization of problem \eqref{problesemidef} and suppose that $h$ is differentiable  at $\bar x$. Then,  there exists $A \in \Sp_+$ with $\Tr(A) =1$ such that  
\begin{align}\label{complcondsemi01}
	0_{X^\ast } \in \hat{\partial } \varphi(\bar x) + \hat{ D}^\ast \Phi(\bar x) (A) + N_\Constr(\bar x), \text{ and } v^\top\Phi(\bar x) v=0  \text{ for each eigenvector }v \text{ of }A, 
\end{align}
provided that the following   qualification   holds
\begin{align}\label{QCsem01}
	0_{X^\ast} \notin \hat{D} \Phi(\bar x)( A) + N_\Constr(\bar x), \text{  for all }    A \in \Sp_+  \text{ with  } \Tr(A) =1.
\end{align}
\end{theorem}

\begin{proof}
First, let us notice that by Proposition \ref{Proposit} the mapping $\Phi$ belongs to $\Gamma_{h}(X,\Sp,\mathcal{B})$, where $\mathcal{B}:=\{  A \in \Sp_+ : \Tr(A) =  1  \}$.  Then, applying Corollary \ref{Cor:CONe:Local:abst:constr} we get the result.
\end{proof}

Consider  normed spaces $X, Y$ and recall that a function $F:X \to Y$ is called $\mathcal{C}^{1,+}$ at $\bar x$ if there exists a neighbourhood $U$ of $\bar x$ such that $ F$  is Fréchet differentiable on $U$ and its gradient is Lipschitz continuous on $U$.

\begin{corollary}
Let $\bar x$ be a  local optimal solution of problem \eqref{problesemidef}. Suppose that  $X$ is a Hilbert space  and    $\varphi$ and that $\Phi $ are $\mathcal{C}^{1,+}$ at $\bar x$.  Then, there exist $v_i \in \R^p$ with $\|v_i\| =1$,  $i=1, \ldots,p$,  and $( \lambda_i)_{i=1}^p \in \Delta_{p}$ such that  $v_i^\top\Phi(\bar x) v_i=0$  and 
{\begin{align*}
	0_{X} \in \nabla \varphi(\bar x) + \sum_{i=1}^p \lambda_i v_i^\top \nabla \Phi (\bar x) v_i+ N_\Constr(\bar x),
\end{align*}}
provided that 
the following  qualification   holds
\begin{align}\label{QCsem02}
	0_{X} \notin \sum_{i=1}^p \lambda_i v_i^\top \nabla \Phi (\bar x) v_i + N_\Constr(\bar x), \text{  for all }   ( \lambda_i)_{i=1}^p \in \Delta_{p} \text{ and }  v_i \in \R^d \text{ with }\|v_i\|=1,
\end{align}
where $  v_i^\top \nabla \Phi (\bar x) v_i$ is the gradient of $x \mapsto v^\top_i \Phi( x) v_i$ at $\bar{x}$.
\end{corollary}
\begin{proof}
Let us consider a closed and convex neighbourhood $U$ of $ \bar{x}$ and $\rho>0$ such that the functions $\varphi(x) + \rho\| x\|^2$ and $\langle A, \Phi(x)\rangle +\rho\| x\|^2 $ are convex over $U$ for all  $  A \in \Sp_+ $ with $\Tr(A) =  1$ (see, e.g, \cite[Proposition 1.11]{MR1016045}). Hence, for $h(x):=\rho \| x\|^2$, $\varphi_U:=\varphi +\delta_{U}\in \Gamma_{h}(X)$. Furthermore, by Proposition \ref{Proposit}  the mapping
\begin{align*} 
	\Phi_U(x) :=\left\{  \begin{array}{cc}
		\Phi(x),& \text{ if } x\in U,\\
		\infty_{\Sp},& \text{ if } x\notin U,
	\end{array}       \right.
\end{align*}
is a DC matrix-mapping with control $h$. Then, it is easy to see that $\bar{x}$ is also a local solution of $\min\{ \varphi_U (x) : \Phi_U(x) \preceq 0, \; x\in \Constr \}$. Let us notice that for every matrix $A \in \Sp_+$, and its spectral decomposition $A= \sum_{i=1}^p \lambda_i u_i u_i^\top$, we get  
\begin{align*}
	\hat{D}^\ast  \Phi(\bar{x})(A)=  \sum_{i=1}^p \lambda_i    u_i^\top \nabla  \Phi(\bar{x}) u_i.  
\end{align*}
Hence,   condition \eqref{QCsem02} implies \eqref{QCsem01}. Therefore,     Theorem \ref{Opt:Cond:semidefi} implies the existence of  $A \in \Sp_+$ with $\Tr(A) =1$ such that \eqref{complcondsemi01} holds. Now, consider $\lambda_i:=\lambda_i(A)$, and associated eigenvalues $v_i$, $i=1,\ldots, p$. Then, $(\lambda_i)\in \Delta_p$  and $  v_i^\top \Phi(x)v_i=0$, and that ends the proof.  
\end{proof}

\section{Conclusions}
The paper deals with optimization problems involving the so-called class of 
\emph{B-DC mappings }(see Definition \ref{definition_DCfunction}), which slightly extend the concept
of delta-convex functions. The most general model studied in the paper is an
optimization problem with an abstract constraint given by a closed convex
set $\C$. The proposed methodology consists in transforming the original
problem into an unconstrained optimization problem (by means of the notion
of improvement function), and in using this reformulation to derive necessary
and sufficient conditions of global and local optimality. The case in which
the abstract constraint is a convex closed cone $-\mathcal{K}$ is discussed
in detail, and global optimality conditions are stated in Theorem \ref{conecons_global}
while Theorems \ref{conecons_local} and \ref{suf_local_cone} deals with local optimality. Our developments
are applied in the last section to establish \emph{ad hoc} optimality
conditions for fundamental problems in applied mathematics such as infinite,
stochastic and semidefinite programming problems. Next, we resume the main
conclusions of the paper:

\begin{enumerate}[label=\arabic*)]
	\item Non-smooth tools like the (regular) subdifferential, the notion of
	(regular) coderivative showed to be appropriate technical instruments in our
	approach, outside of the scope of Asplund spaces.
	
	\item New qualification conditions, which are an alternative to the Slater
	condition, are introduced in the paper. These conditions require certain
	degree of continuity of the objective/constraints functions and ($w^{\ast }$%
	)-compactness of the set $\C^\circ$.
	
	\item Some properties of the B-DC mappings are supplied by Proposition \ref{PropoBasicPro}.
	
	\item Theorem \ref{Theo:aprxsubd} is a key result in our analysis. It is based on
	Proposition \ref{Epsilonformula}, a useful characterization of the $\varepsilon $-subdifferential of the supremum of convex functions.
	
	\item The particular structure of the cone-constraint problem allows us to
	build more suitable supremum functions. This is the case when the polar cone 
	$\mathcal{K}^{+}$ is $w^{\ast }-$compactly generated, and a representative
	example of that situation is the semi-infinite optimization model where $%
	\mathcal{K}^{+}$ is the set of all (finite) regular Borel measures.
	
	\item In Proposition \ref{pisub}, a classical formula for the subdifferential of
	convex normal integrand functions is extended to the case of nonconvex
	normal integrands. 
	
	\item  In the last subsection, devoted to semidefinite programming, the
	notion of DC-matrix mapping is introduced. This concept leads to
	the main associated optimality result, which is Theorem \ref{Opt:Cond:semidefi}.
\end{enumerate}


\bibliographystyle{plain}
\bibliography{references}

\end{document}